\documentclass[12pt]{amsart}
\usepackage{bbm} 
 \usepackage{graphicx}
\usepackage{amsmath,amsthm,amsfonts,amscd,amssymb,mathrsfs,hyperref}
\usepackage[all]{xypic}
\usepackage[enableskew]{youngtab} 
\usepackage{multicol}	

\numberwithin{equation}{section}

\usepackage{color}
\newtheorem{theorem}{Theorem}[section]

\newtheorem{lemma}[theorem]{Lemma}

\newtheorem{prop}[theorem]{Proposition}

\newtheorem{cor}[theorem]{Corollary}

\newcommand{\defi}[1]{\textsf{#1}} 

\newcommand{\SL}{\operatorname{SL}}

\newcommand{\V}{{\mathcal V}}
\newcommand{\I}{\mathcal{I}}

\newcommand{\J}{\mathcal{J}}

\newcommand{\PP}{\mathbb{P}}

\newcommand{\NN}{\mathbb{N}}
\newcommand{\CC}{\mathbb{C}}
\newcommand{\ZZ}{\mathbb{Z}}

\newcommand{\Seg}{\operatorname{Seg}}

\def\bw#1{{\textstyle\bigwedge^{\hspace{-.2em}#1}}}
\def\o{ \otimes }
\def \P{\mathcal{P}}

\theoremstyle{definition}
\newtheorem{definition}[theorem]{Definition}
\newtheorem{example}[theorem]{Example}

\theoremstyle{remark}
\newtheorem{remark}[theorem]{Remark}

\newcommand{\sgn}{{\rm sgn}}

\newcommand{\mtx}[1]{\left(\begin{smallmatrix} #1 \end{smallmatrix}\right)}	
\newcommand{\ds}{\displaystyle}
\newcommand{\MS}{{\rm MS}}

\begin{document}
\author{Huajun Huang}\email{huanghu@auburn.edu}
\author{Luke Oeding}\email{oeding@auburn.edu}
\address{Department of Mathematics and Statistics,
Auburn University,
Auburn, AL, USA
}

\date{\today}

\title{Symmetrization of Principal Minors and Cycle-Sums}
\begin{abstract}
We solve the Symmetrized Principal Minor Assignment Problem, that is we show how to determine if for a given vector $v\in \CC^{n}$ there is an $n\times n$ matrix that has all $i\times i$ principal minors equal to $v_{i}$. 
We use a special isomorphism (a non-linear change of coordinates to cycle-sums) that simplifies computation and reveals hidden structure. We use the symmetries that preserve symmetrized principal minors and cycle-sums to treat 3 cases:  symmetric, skew-symmetric and general square matrices. 
We describe the matrices that have such symmetrized principal minors as well as the ideal of relations among symmetrized principal minors / cycle-sums.  
 We also connect the resulting algebraic varieties of symmetrized principal minors to tangential and secant varieties, and Eulerian polynomials. 
\end{abstract}
\maketitle

\section{Introduction}
The famous Principal Minor Assignment Problem (PMAP) asks to find a matrix, (or determine when one exists), with a prescribed set of values for its principal minors. 
Numerical solutions have been proposed in \cite{GriffTsat2} and in the case of symmetric matrices in \cite{Rising2014}.

To solve this problem algebraically, we would like to find a minimal generating set of the ideal of relations among the principal minors. The $4\times 4$ case was solved in \cite{BorodinRains, LinSturmfels}, whereas the $n\times n$ case for $n\geq 5$ is widely open. The case for symmetric matrices was solved for $n=3,4$ in \cite{HoltzSturmfels} and set-theoretically for all $n$ in \cite{Oeding_principal, oeding_thesis}. The ideal-theoretic version (in the symmetric case) for arbitrary $n$ remains open, however see \cite{KP14} for a recent approach using cluster relations to study principal minors and almost principal minors. 
  We also note that the question of finding relations among minors of a fixed size is also an interesting problem, but it is quite difficult (see \cite{BrunsConcaVarbaro}).
Grinshpan et al. \cite{Woerdeman_pm} studied the principal minor problem when all principal minors of a given size are equal (the \emph{symmetrized principal minors property}) in relation to the question of determinantal representations of multivariate polynomials.

In Section~\ref{sec:spm} we review principal minors and explain the type of symmetrization which leads to the ``principal minors of equal size are all equal'' condition.
In Section~\ref{sec:cycles} we explain the non-linear isomorphism on affine space to cycle-sum coordinates. It turns out that it is much easier to compute the ideals in which we are interested using cycle-sums. In addition, in cycle-sum coordinates, one of the embedded components in the ideal of symmetrized cycle-sums is a line. This structure is not apparent in principal minor coordinates.

In Section~\ref{sec:Matrices} we classify (up to diagonal and permutation conjugation symmetry) the matrices that have symmetrized principal minors and (equivalently) symmetrized cycle-sums.
We also provide a minimal parametrization of the respective varieties of symmetrized principal minors in the cases of symmetric, skew-symmetric and general square matrices.
In Section~\ref{sec:polynomials} we study the defining ideals of these varieties.

If $A$ has symmetrized cycle-sums, we use $c_{k}:= c_{k}(A)$ to denote the order-$k$ cycle-sum, and $d_{k}:=d_{k}(A)$ to denote the order-$k$ principal minor. Here is a summary of what we found:
\begin{theorem}\label{thm:Matrices}
Suppose $A \in \CC^{n\times n}$ has symmetrized principal minors and $n\geq 2$.
\begin{enumerate}
\item If $A$ is symmetric, then $A$ is conjugate to 
\[
\lambda \mathbbm{1}_{n} + \mu I_{n} , \quad \text{for} \;\; \lambda, \mu \in \CC,
\]
where $\mathbbm{1}_{n}$ denotes the $n\times n$ all-ones matrix. We have the following parameterizations:
\begin{eqnarray*}
d_{k}(\lambda \mathbbm{1}_{n} + \mu I_{n}) &=& \mu^{k-1}\cdot (\mu+k\cdot \lambda) 
\;,\\
c_{k}( \lambda \mathbbm{1}_{n} + \mu I_{n}) &=& (k-1)! \cdot \lambda^{k}
\;.\end{eqnarray*}
\item If $A$ is skew-symmetric, $A$ is conjugate to
\[
 \lambda \mathbbm{1}_{n}^{\wedge}\;,
\; \text{ or }\;\;
\lambda \mtx{
  0&  1&   1& 1\\
   -1&  0&  1&  -1\\
   -1&  -1&   0&  1\\
   -1&  1&  -1&  0\\
} \text{ (for } n=4 \text{ only)},
\qquad\text{for}\quad \lambda\in\CC,\]
where $\mathbbm{1}_{n}^{\wedge}$ denotes the $n\times n$ skew-symmetric matrix with 1's above the diagonal.
We have the following parameterizations:
\[
d_{k}(\mathbbm{1}_{n}^{\wedge}) = 1 \quad \text{for} \;\; k\geq 2 \text{ and }   k-even
,\]
\[
c_{k}(\mathbbm{1}_{n}^{\wedge}) = (-1)^{s/2}E_{k-1}, \quad \text{ where } E_{k} \text{ is the Euler number.}  
\]
\item If $A$ is general, then
\begin{enumerate}
\item If
$n\ge 3$, and $c_1=c_2 =0$,
then one of the following holds
\begin{enumerate}
\item 
$A$ is conjugate  to a strictly upper triangular matrix,
 where
\[c_1=c_2=\cdots=c_n=0.\]
\item 
$A$ is conjugate to a matrix representing an $n$-cycle and
\[c_1=c_2=\cdots=c_{n-1}=0,\qquad c_n\ne 0.\]
\end{enumerate}

\item 
If $c_2\ne 0$ and $c_{1}=c_3=0$, then $A$ is conjugate to a skew-symmetric  matrix with symmetrized principal minors.
\item If $c_1=0$, and $c_2c_{3}\ne 0$, then $A$ is  conjugate to $\lambda T_{n}(x)$, where $T_{n}(x)$ is  the following Toeplitz matrix for $x\in\CC^{*}$:
\[
T_n(x):=
\mtx{
      0 &1 &x &x^2 &\cdots &x^{n-2}
\\ -1 &0 &1 &x &\cdots &x^{n-3}
\\ -\frac{1}{x} &-1 &0 &1 &\cdots &x^{n-4}
\\ -\frac{1}{x^{2}} &-\frac{1}{x} &-1 &0 &\cdots &x^{n-5}
\\ \vdots &\vdots &\vdots &\vdots &\ddots &\vdots
\\ -\frac{1}{x^{n-2}} &-\frac{1}{x^{n-3}} &-\frac{1}{x^{n-4}} &-\frac{1}{x^{n-5}} &\cdots &0
},
\]
where the  $(i,j)$ entry of $T_n(x)$ is exactly $\sgn(j-i)\cdot x^{j-i-\sgn(j-i)}$. Moreover
    $\lambda^2=-c_2$ and $\lambda^3(x-\frac{1}{x})=c_3$, and
\[c_{s}(T_{n}(x)) \quad =\quad x^{-s} E_{s-1}(-x^{2}),\] where  $E_{n}(x)$ is the $n$-th Eulerian polynomial.

Also \quad $
d_{s} (T_{n}(x))
=
\frac{(x^{2})^{s-1}+(-1)^{s}}{x^{s-2}(x^{2}+1)}
$ so 
$
(x^{2}+1)d_{s}(x\cdot T_{n}(x)) =
  x^{2s}+(-1)^{s}x^{2}
  $.

\end{enumerate}
\end{enumerate}
\end{theorem}
We prove Theorem~\ref{thm:Matrices}  and give the explicit conjugations (via permutation and diagonal matrices) in each separate case in Section~\ref{sec:Matrices}.
Note that Theorem~\ref{thm:Matrices} describes the pull-back of the symmetrization of principal minors (cycle-sums) conditions to matrices (SCS matrices). In particular, the set of SCS matrices is reducible and each component is the base of a parametrization of a (possibly) different component in the target. We find that each of the components of the source map to the same irreducible base.

Simply applying elimination to the ideal-theoretic problem produces an embedded scheme with non-reduced structure. Our methods show that the embedded component is supported on the line $\V(c_{2},\ldots,c_{n})$ and seems to have complicated scheme structure (apart from exceptional initial cases which we describe). An interesting avenue for future study would be to investigate the non-reduced structure. See Example~\ref{ex:J4} and Remark~\ref{rmk:structure}.

In Sections~\ref{sec:sym relations},~\ref{sec:skew relations}, and \ref{sec:nonsym relations} we respectively describe the ideals of these (geometric) components in the cases of symmetric, skew-symmetric, and general matrices. Here is a summary:

\begin{theorem}\label{thm:main2}
Let $Z_{n}$, (respectively $Z_{n}^{\circ}$ and $Z_{n}^{\wedge}$) denote the variety of cycle-sums of $n\times n$ general (respectively symmetric, skew-symmetric) matrices for $n\geq 3$. 
Let $\J_{n}$ (respectively $\J^{\circ}_{n}$, $\J^{\wedge}_{n}$) denote the ideal of the (set-theoretic) intersection of $Z_{n}$ (respectively $Z_{n}^{\circ}$ and $Z_{n}^{\wedge}$) and the linear space of symmetrized cycle-sums. 

\begin{enumerate}
\item
 If $n=3$ then $\J^{\circ}_{n}$ is the principal ideal
 \[
 \J^{\circ}_{3}=
 \langle -4 c_{2}^{3}+c_{3}^{2} \rangle 
 .\]
For $n\geq 4$ $\J^{\circ}_{n}$ is
the prime ideal generated by the following $n-2$ binomials:
\[
\left\{4 c_{2}^{3}-c_{3}^{2}\right\} \cup
\left\{ (s-1)!c_{2}c_{s-2} - (s-3)!c_{s} \;\mid\; 4\leq s\leq n
\right\}.
\]
\item  
$\J_{3}^{\wedge} = \langle c_{1},c_{3}\rangle$.
$\J^{\wedge}_{4}$ decomposes as the intersection of two prime components
\[\J^{\wedge}_{4} \;=\;
\langle
-2 c_{2}^{2}+c_{4}, c_{1},c_{3}
\rangle\quad \cap \quad \langle
6 c_{2}^{2}+c_{4},c_{1},c_{3}
 \rangle.\]
For $n\geq 5$, 
$\J^{\wedge}_n$ is 
the prime ideal generated by 
\[\{c_{2k-1}\mid 1\leq k \leq \lfloor n/2\rfloor\}
\cup
\{ E_{2(i+j)-1}c_{2i}c_{2j} -E_{2i-1}E_{2j-1} c_{2(i+j)} \mid 1\leq i \leq j \leq \left\lfloor \frac{n}{2} \right\rfloor, i+j \leq n
\}.\]
\item 
$\J_{3}$ is empty.
$\J_{4}$ decomposes as the intersection of two prime components:
\[\langle 2c_2^3+c_3^2-c_2c_4\rangle\qquad \text{and} \qquad \langle c_3,6c_2^2+c_4 \rangle.\]
When $n\geq 5$,
$\J_{n}$ is 
 the prime ideal generated by the maximal minors of
\[
\begin{pmatrix}
d_{0}       &  d_{1}       & d_{2} &\dots & d_{n-2}\\	
d_{1}       &  d_{2} &d_{3} &\dots &d_{n-1} \\
d_{2} & d_{3}& d_{4} &\dots & d_{n}\\
\end{pmatrix}.\]
\end{enumerate}
\end{theorem}

Symmetrization of principal minors produces a scheme that may seem mysterious at first  
in principal minor coordinates. However, in cycle-sum coordinates one sees an embedded scheme supported on a line. Removing this embedded component, the geometry is revealed:

\begin{theorem}  Let $\varphi$ denote the projectivized principal minor map $\varphi\colon \CC^{n\times n} \oplus \CC \to \PP S^{n}\CC^{2}$.
\begin{enumerate}
\item For $n\geq 3$,
$\varphi( S^{2}\CC^{n}\oplus \CC) \cap S^{n}\CC^{2}  = \left(\tau \nu_{n}\PP^{1} \cap U_{d_{0}=1} \right)\cup \left( \mathcal{V}(c_{2},\dots,c_{n})\cap U_{c_{0}=1}\right)$;
 the tangential variety to the Veronese  and an embedded scheme supported on a line.

\item If $n=4$, then $\varphi( \bw{2}\CC^{4}\oplus \CC) \cap S^{4}\CC^{2}  = \nu_{4}\PP^{1}  \cap \mathcal{V}(d_{0}-1,d_{1},d_{3})$ together with the curve  $\{[(1,0,b,0,9b^{2})]\mid b\in \CC\}$ (in symmetrized principal minor coordinates).

\noindent If $n\geq 5$,  
$\varphi( \bw{2}\CC^{n}\oplus \CC) \cap S^{n}\CC^{2} = \nu_{n}\PP^{1}\cap \mathcal{V}(\{d_{0}-1\}\cup\{ d_{2k-1} \mid 1\leq k \leq n/2 \}  )$;  
a linear section of the Veronese variety.

\item If $n=4$, then  $\varphi(  \CC^{n\times n}\oplus \CC) \cap S^{n}\CC^{2}  = \sigma_{2}\nu_{4}\PP^{1} \cup  \{ [1,0,b,0,9b^{2}] \mid b\in \CC\}$. 

\noindent If $n\geq 4$, 
$\varphi( \CC^{n\times n}\oplus \CC) \cap S^{n}\CC^{2}  = \sigma_{2} \nu_{n}\PP^{1} \cap U_{d_{0}=1} \cup \mathcal{V}(c_{2},\dots,c_{n})\cap U_{c_{0}}=1$;
 the secant variety to the Veronese  and an embedded scheme supported on a line.
\end{enumerate}

\end{theorem}
\begin{remark}
  See \cite{OedingRaicu, RaicuGSS} for the ideals of tangential and chordal varieties to Segre-Veronese varieties in full generality. 
\end{remark}
\begin{remark}
Our results also provide insights into the general PMAP. In particular, by casting this problem as a symmetrization of the PMAP, we know that the solution to the general PMAP must symmetrize to the solution we provide in this work. For example, in the case of $3\times 3$ symmetric matrices the variety of principal minors is isomorphic to the tangential variety of the Segre product $ \tau(\Seg(\PP^{1}\times \PP^{1}\times \PP^{1}))$, which symmetrizes to the tangential variety of the Veronese variety $\tau(\nu_{3}\PP^{1})$, the main component of the variety symmetrized principal minors of symmetric matrices.
\end{remark}
\section{Principal minors, cycle-sums, and symmetrization}\label{sec:spm}
\subsection{Principal minors and symmetrized principal minors}
Let $A = (a_{i,j})$ be an $n\times n$ matrix. For $S\in \mathcal{P} (n)$ (the set of subsets of $[n]$), let $D_{S}(A)$ denote the principal minor of $A$ with row and column set $S$. The functions $\{D_{S} \mid S \in \P(n)\}$ furnish \emph{principal minor coordinates} on $\CC^{\P(n)}$.
Consider the principal minor map:
\begin{eqnarray*}
\phi\colon \CC^{n\times n} &\to& \CC^{\P(n)}\\
 A &\mapsto& (D_{S}(A))_{S\in \mathcal{P}(n)}
, \end{eqnarray*}
where we may assume $D_{\emptyset}(A)=1$. Let $Z_{n}$ denote the image $\phi(\CC^{n\times n})$, the \emph{variety of principal minors of $n\times n$ matrices}, which is closed by \cite[Thm.~1]{LinSturmfels}.

 We will also restrict the domain to symmetric matrices ($S^{2}\CC^{n}$) and skew-symmetric matrices ($\bw{2}\CC^{n}$). Let $Z^{\circ}_{n}:=\phi(S^{2}\CC^{n})$, and $Z^{\wedge}_{n}:=\phi(\bw{2}\CC^{n})$. The reader may wish to consult \cite{LandsbergTensorBook} for an introduction to tensors from a geometric viewpoint.

We may identify the space $\CC^{\mathcal{P}(n)}$ with $\CC^{2}\o\dots\o\CC^{2} = \CC^{2^{n}}$, which reflects the $\SL(2)^{\times n}\rtimes \mathfrak{S}_{n}$ symmetry of the target space that preserves the projective variety parametrized by $\phi$, (see \cite[Theorem~1.1]{Oeding_principal} or \cite[Theorem 15]{HoltzSturmfels}.) 
Now consider the subspace $\CC^{n+1}$ in $\CC^{\P(n)}$ defined by the condition that $D_{S} = D_{S'}$ for all $S,S' \subset\P(n)$ such that $|S| = |S'|$, and let $d_{k}$ denote the value of $D_{S}$ when $|S|=k$. Viewed as a subspace of $\CC^{2}\o\cdots\o\CC^{2} = \CC^{2^{n}}$, we see that this copy of $\CC^{n+1}$ is naturally isomorphic to $S^{n}\CC^{2}$, the space of fully symmetric $2\times \dots \times 2$ tensors. Note that $S^{n}\CC^{2}$ is naturally an $\SL(2)$-module (and a trivial $\mathfrak{S}_{n}$-module), and the copy of $\SL(2)$ acting on $S^{n}\CC^{2}$ is the diagonal copy in $\SL(2)^{\times n}$ acting on $(\CC^{2})^{\o n}$. We will let $\{d_{i} \mid 0\leq i \leq n\}$ denote \emph{symmetrized principal minor coordinates} on $S^{n}\CC^{2}$.

The \emph{variety of symmetrized principal minors}, denoted $Y_{n}$, is the variety of principal minors of $n\times n$ matrices whose principal minors of equal size have the same value. Geometrically we have the intersection 
$
Y_{n}:= Z_{n} \cap S^{n}\CC^{2}
.$
Analogously define $Y^{\circ}_{n}:=Z^{\circ}_{n}\cap S^{n}\CC^{2}$ and $Y^{\wedge}_{n}:=Z^{\wedge}_{n}\cap S^{n}\CC^{2}$ respectively in the skew-symmetric and symmetric cases.

This symmetrization process was studied in the context of hyperdeterminants in \cite{Oeding_hyperdet}, and the first example of symmetrization of principal minors happens to coincide with the first example of the symmetrization of hyperdeterminants: 

\begin{example}\label{ex:discrim}
Consider the case of $3\times 3$ symmetric matrices. Holtz and Sturmfels showed that $Z^{\circ}_{3}$ is a hypersurface defined by  Cayley's $2\times 2\times 2$ hyperdeterminant 
\[\begin{smallmatrix} 
Det\; =\; D_{\emptyset}^{2}{D}_{\{1,2,3\}}^{2}
+{D}_{\{1\}}^{2} {D}_{\{2,3\}}^{2}     +{D}_{\{2\}}^{2} {D}_{\{1,3\}}^{2}+ {D}_{\{3\}}^{2}{D}_{\{1,2\}}^{2}  
+4 \left(
{D}_{\{1\}} {D}_{\{2\}} {D}_{\{3\}} {D}_{\{1,2,3\}}   + D_{\emptyset}{D}_{\{1,2\}} {D}_{\{1,3\}} {D}_{\{2,3\}} \right)
\\
-2\left(\begin{smallmatrix}
   {D}_{\{1\}} {D}_{\{2\}} {D}_{\{1,3\}} {D}_{\{2,3\}}   
+ {D}_{\{1\}} {D}_{\{1,2\}} {D}_{\{3\}} {D}_{\{2,3\}}
+ {D}_{\{2\}} {D}_{\{1,2\}} {D}_{\{3\}} {D}_{\{1,3\}}
\\
+ D_{\emptyset}{D}_{\{1\}} {D}_{\{2,3\}} {D}_{\{1,2,3\}}
+ D_{\emptyset}{D}_{\{2\}} {D}_{\{1,3\}} {D}_{\{1,2,3\}}
+ D_{\emptyset}{D}_{\{1,2\}} {D}_{\{3\}} {D}_{\{1,2,3\}}
\end{smallmatrix} \right)
.\end{smallmatrix}
\]
The symmetrization of the $2\times 2\times 2$ hyperdeterminant (setting $D_{S} = d_{|S|}$ and $D_{\emptyset} = 1$),
 \[
SDet = -3 {d}_{1}^{2} {d}_{2}^{2}+4 {d}_{1}^{3} {d}_{3}+4 {d}_{2}^{3}-6 {d}_{1} {d}_{2} {d}_{3}+{d}_{3}^{2}
 ,\]
  is the discriminant of the cubic $
1+ 3d_{1}x + 3d_{2}x^{2} + d_{3}x^{3}
$, (see \cite[Sec.~3.6]{SturmfelsAlgorithms}).
The ideal of $Y^{\circ}_{3}$ is minimally generated by $SDet$. 
In cycle-sums $C_{I}$ (see Def.~\ref{def:cycle-sums}) the $2\times 2\times 2$ hyperdeterminant is
\[
Det =  -4 C_{\{1,2\}} C_{\{1,3\}} C_{\{2,3\}}+C_{\{1,2,3\}}^{2}
,\]
which is the same formula as \cite[Sec. 2, eq. (8)]{SturmfelsZwiernik} since, in this case, cycle-sums correspond to binary cumulants.  In symmetrized cycle-sums, the symmetrized hyperdeterminant becomes 
\[
SDet = -4c_{2}^{3} + c_{3}^{2}
,\]
and since
\[c_2 \;=\; d_1^2-d_2,  \qquad c_3 \;=\; 2d_1^3-3d_1d_2+d_3,\]
and the fact that we have set the constant term equal to $1$,  $SDet$ is the same expression as the syzygy amongst the covariants of the binary cubic.
\end{example}

\subsection{Cycle-sums and symmetrized cycle-sums}\label{sec:cycles}
We are interested in studying the relations among principal minors of different types of matrices. In this section, we will explain a special non-linear change of coordinates (to cycle-sums) that simplifies the relations. 
The connection between cycle-sums and principal minors is illuminated by the combinatorics governed by the underlying geometric lattice and its M\"obius function, whose properties are well-explained in Stanley's book \cite[Ch.~3, Ex. 3.10.4]{Stanley}.  Stanley's historical notes attribute these results to independent discoveries by Sch\"utzenberger, and Rota and Frucht.

The idea to look at cycle-sums in their connection to the relations amongst principal minors appeared previously in the work of Lin and Sturmfels \cite{LinSturmfels}, 
 and has the same theme as some work of Rota \cite{RotaCumulants}. Sturmfels and Zwiernik's work on binary cumulants (see Ex.~\ref{ex:discrim}) showed, in particular, that the $2\times 2\times 2$ hyperdeterminant is a binomial in cumulant coordinates \cite{SturmfelsZwiernik}. 
Michalek, Zwiernik, and the second author introduced \emph{secant cumulants}, which reveal toric structure on the secant and tangential varieties to the Segre variety \cite{MichalekOedingZwiernik}. Manivel and Michalek used similar methods to study minuscule and cominuscule varieties, \cite{ManivelMichalek}.

\begin{definition}\label{def:cycle-sums}
For $A \in \CC^{n\times n}$ and $I\subset [n]$ the \defi{cycle-sum} $C_{I}$ is defined by the following:
\[
C_{I}(A) := \sum_{\{i_{1},\dots,i_{k}\}=I,\quad i_{1}=\min I } a_{i_{1},i_{2}} a_{i_{2},i_{3}}\cdots a_{i_{k-1},i_{k}}a_{i_{k},i_{1}}.
\]
Note if $i_{1} = \min I$, the sum is over all permutations of $\{i_{2},\dots,i_{k}\}$. 
We will set $C_{\emptyset}(A)=1$.
\end{definition}

\begin{example}The first few cycle-sums are the following.
\[\begin{matrix}
C_{\emptyset}(A) &=& 1,\hfill\\
C_{\{1\}}(A) &=& a_{1,1} ,\hfill\\
C_{\{1,2\}}(A) &=& a_{1,2}a_{2,1} ,\hfill\\
C_{\{1,2,3\}}(A) &=& a_{1,2}a_{2,3}a_{3,1} + a_{1,3}a_{3,2}a_{2,1},\hfill\\
C_{\{1,2,3,4\}}(A) &=&
\begin{smallmatrix}
  a_{1,2}a_{2,3}a_{3,4}a_{4,1} + a_{1,3}a_{3,2}a_{2,4}a_{4,1} + a_{1,4}a_{4,2}a_{2,3}a_{3,1}  
 +a_{1,2}a_{2,4}a_{4,3}a_{3,1} + a_{1,3}a_{3,4}a_{4,2}a_{2,1} + a_{1,4}a_{4,3}a_{3,2}a_{2,1}
\end{smallmatrix}
\end{matrix}\]
\end{example}

The functions $\{C_{S} \mid S\in \P(n)\}$ furnish  \defi{cycle-sum coordinates} on $\CC^{\P(n)}$. We define the cycle-sum map 
\begin{eqnarray*}
\varphi\colon\CC^{n\times n} &\to& \CC^{\P(n)}\\
 A &\mapsto& ( C_{S}(A))_{S\in \mathcal{P}(n)}
\end{eqnarray*}
Let $X_{n}$ denote the image $\varphi(\CC^{n\times n})$, the variety of cycle-sums of $n\times n$ matrices, and similarly define $X^{\circ}_{n}$ and $X^{\wedge}_{n}$ to be the analogous varieties of cycle-sums of symmetric and skew-symmetric $n\times n$ matrices.

\begin{prop}
Suppose $A\in \CC^{n\times n}$ is such that $D_{I}(A) = D_{J}(A)$ whenever $|I|= |J|$. Then $C_{I}(A) = C_{J}(A)$ whenever $|I|= |J|$.
\end{prop}
\begin{proof}
Proof by (easy) induction.
\end{proof}
We also have the following useful fact:
\begin{prop}[{\cite[Cor.~5]{LinSturmfels}}]
A vector $u_{*} \in \CC^{2^{n}}$ is realizable as the principal minors of an $n\times n$ matrix if and only if the corresponding vector under the isomorphism in Prop.~ \ref{thm:d_S=c_S} is realizable as the cycle-sums of an $n\times n$ matrix.
\end{prop}

\subsection{Transition between principal minors and cycle-sums} 
\begin{example}\label{ex-cd-[4]}By direct calculation one can find the transitions between $D_{S}$ and $C_{S}$ coordinates. For instance
 $D_{\{1,2,3,4\}}$ and $C_{\{1,2,3,4\}}$ transform as follows:
\[\begin{smallmatrix}
 D_{\{1,2,3,4\}} &=& -C_{\{1,2,3,4\}}
 ,\hfill\\
&& +C_{\{1,2,3\}}C_{\{4\}}+C_{\{1,2,4\}}C_{\{3\}}+C_{\{1,3,4\}}C_{\{2\}}+C_{\{2,3,4\}}C_{\{1\}}
 ,\hfill\\
&& +C_{\{1,2\}}C_{\{3,4\}}+C_{\{1,3\}}C_{\{2,4\}}+C_{\{1,4\}}C_{\{2,3\}}
 ,\hfill\\
&& -\left(
C_{\{1,2\}}C_{\{3\}}C_{\{4\}}+C_{\{1,3\}}C_{\{2\}}C_{\{4\}}+C_{\{1,4\}}C_{\{2\}}C_{\{3\}} 
+C_{\{2,3\}}C_{\{1\}}C_{\{4\}}+C_{\{2,4\}}C_{\{1\}}C_{\{3\}}+C_{\{3,4\}}C_{\{1\}}C_{\{2\}} \right)
 ,\hfill\\
&& +C_{\{1\}}C_{\{2\}}C_{\{3\}}C_{\{4\}}
\;, \hfill\\
 C_{\{1,2,3,4\}} &=& -D_{\{1,2,3,4\}}
 ,\hfill\\
&& +D_{\{1,2,3\}}D_{\{4\}}+D_{\{1,2,4\}}D_{\{3\}}+D_{\{1,3,4\}}D_{\{2\}}+D_{\{2,3,4\}}D_{\{1\}}
 ,\hfill\\
&& +D_{\{1,2\}}D_{\{3,4\}}+D_{\{1,3\}}D_{\{2,4\}}+D_{\{1,4\}}D_{\{2,3\}}
 ,\hfill\\
&& 
-2\left(
D_{\{1,2\}}D_{\{3\}}D_{\{4\}}+D_{\{1,3\}}D_{\{2\}}D_{\{4\}}+D_{\{1,4\}}D_{\{2\}}D_{\{3\}}
D_{\{2,3\}}D_{\{1\}}D_{\{4\}}+D_{\{2,4\}}D_{\{1\}}D_{\{3\}}+D_{\{3,4\}}D_{\{1\}}D_{\{2\}}
\right)
 ,\hfill\\
&& +6D_{\{1\}}D_{\{2\}}D_{\{3\}}D_{\{4\}}. \hfill 
\end{smallmatrix}
\]

\end{example}

In this section, we give general formulas for the transition between principal minor coordinates and cycle-sum coordinates. The key, like in the case of cumulants \cite[Ch.~4]{ZwiernikBook}, is to notice that our coordinates are indexed by the elements of a nice poset, whose M\"obius function and rank functions we know and can use for the changes of coordinates. 
We follow Stanley's notation, and \cite[Sec.~3.7, Sec.~3.10]{Stanley}, especially  \cite[Example~3.10.4]{Stanley}. Recall that if $S\subset [n]$, the set of all \defi{set-partitions} of $S$ is
\[
\Pi_S:=\left\{\{S_{1},S_{2},\cdots,S_{k}\}\mid k\in\NN,\ S_{1}\sqcup S_{2}\sqcup\cdots\sqcup S_{k}=S,\ S_{i}\ne\emptyset \text{ for } i\in [k]\right\}.
\] 
Abbreviate the partition $\{S_{1},S_{2},\cdots,S_{k}\}$  as $S_1 S_2\cdots S_k$. 
The set-partitions on a set $S$, denoted  $\Pi_{S}$,  are partially ordered by refinement $\succeq$.
The poset $\Pi_{S}$ is a lattice, with rank
\[
\rho(S_1 S_2\cdots S_k):= (|S_1|-1)+(|S_2|-1)+\cdots +(|S_k|-1)=|S|-k.
\]
and sign
\[
\sgn(S_1 S_2\cdots S_k):= (-1)^{\rho(S_1 S_2\cdots S_k)}
.\]
When $S=\{i_{1},i_{2},\cdots,i_{s}\}$, $\Pi_S$ has unique maximal and minimal elements $S$ and $\{i_1\}\{i_2\}\cdots \{i_s\}$ respectively.

Analogously, let  $\Pi_{s}$ denote the poset of all partitions of $s\in\ZZ^+$ ordered by refinement: 
The elements of $\Pi_{s}$ may be expressed as $\alpha=a_1^{t_1} a_2^{t_2}\cdots a_{\ell}^{t_\ell}$, where $a_i, t_i\in\ZZ^+$, $a_1>a_2>\cdots>a_{\ell}$,  and $\sum_{i=1}^{\ell} {t_ia_i}=s$. Let $\# \alpha := \sum_{i}t_{i}$ denote the number of parts of $\alpha$.
Then $s^1$ is the maximal element, and $1^s$ is the minimal element, of $\Pi_{s}$ respectively. 
The rank and sign of $\alpha\in\Pi_{s}$  are
\[\rho(\alpha) := s - \# \alpha, \quad
\sgn(\alpha):=(-1)^{\rho(\#\alpha)}.\]
For $S_{1} S_{2}\cdots  S_{k}\in\Pi_S$ , \defi{shape} is the partition
\[
|S_1 S_2\cdots S_k|:=\{|S_1|,|S_2|,\cdots,|S_k|\} \in \Pi_{|S|}=\Pi_{s}.
\] 
The \defi{type} of the partition $S_{1} S_{2}\cdots  S_{k}$ is $(m_{1},\ldots, m_{s})$, where $m_{i}$ is the number of blocks of size $i$ for $1\leq i\leq s$.
For $\alpha\in\Pi_{s}$, let $p_{\alpha}$ denote the number of set-partitions of $S$ of the same type as $\alpha$.  As recorded in \cite[Eqn.~3.36]{Stanley} we have
\begin{equation}\label{palpha}
p_{\alpha} = \frac{s!}{1!^{m_{1}} m_{1}! 2!^{m_{2}} m_{2}! \cdots s!^{m_{s}} m_{s}!}.
\end{equation}
The lattice of set-partitions $\Pi_{s}$ has M\"obius function determined by \cite[Eq.3.37]{Stanley} 
\[
\mu_{s} =  (-1)^{s-1}(s-1)!.
\]

\noindent The key observation is the following isomorphism of coordinate rings. 

\begin{prop}[{\cite[Prop.~4]{LinSturmfels}}]
\label{thm:d_S=c_S}
 Fix $n\in\ZZ^+$.
Consider rings $R_C = \CC\left[C_{S}\mid S\in \P(n)\right]$ and $R_{D} = \CC\left[D_{S}\mid S\in \P(n)\right]$. We have a (lower triangular) non-linear isomorphism of rings
$R_{D} \to R_C$
given by $D_{\emptyset}=1$ and
\begin{equation}\label{d_S=c_S}
D_{S}=\sum_{ S_{1} S_{2} \cdots S_{k}\in\Pi_S} (-1)^{|S|-k}\; C_{S_{1}}C_{S_{2}}\cdots C_{S_{k}}.
\end{equation}
Conversely, we  have a (lower triangular) non-linear isomorphism of rings
$R_C \to R_{D}$
given by 
\begin{equation}\label{c_S=d_S}
C_{S}=\sum_{S_{1} S_{2} \cdots S_{k}\in\Pi_S}  (-1)^{|S|-k}(k-1)!\;  D_{S_{1}}D_{S_{2}}\cdots D_{S_{k}}
.\end{equation}
\end{prop}
\begin{proof}[Lin and Sturmfels' proof]
The transition $R_{D} \to R_C$ is Leibnitz's formula.
The transition $R_C\to R_{D}$ follows by  M\"obius inversion \cite[Prop.~3.7.1]{Stanley} on the lattice of set-partitions. 
\end{proof}

\subsection{Transition between symmetrized principal minors and cycle-sums}

Let $\{c_{i} \mid 0\leq i\leq n\}$ denote \emph{symmetrized cycle-sum coordinates} on $S^{n}(\CC^{2})$. 
The isomorphism between the cycle-sum ring and the principal minor ring descends to the symmetrized case:

\begin{example}\label{ex:sym_cd}Here are the first few cases of the isomorphism and its inverse:

\begin{tabular}{l|l}$\begin{matrix}
d_1&=&c_1  ,\hfill\\
d_2&=&c_1^{2}-c_2  ,\hfill\\
d_3&=&c_1^{3} - 3c_1c_2 + c_3  ,\hfill\\
d_4&=&c_1^4-6c_1^2c_2+3c_2^2+4c_1c_3-c_4 ,\hfill\\
d_{5}&=&c_1^5-10c_1^3c_2+15c_1c_2^2+10c_1^2c_3 \hfill
\\&& -10c_2c_3-5c_1c_4+c_5 ,\hfill\\
d_{6}&=&\begin{smallmatrix}c_1^6-15c_1^4c_2+45c_1^2c_2^2+20c_1^3c_3-15c_2^3-60c_1c_2c_3 \\
 -15c_1^2c_4+10c_3^2+15c_2c_4+6c_1c_5-c_6\;.  \hfill \end{smallmatrix}
\end{matrix}$&$
\begin{matrix}
c_1&=& d_1,\hfill\\
c_2 &=& d_1^2-d_2,\hfill\\
c_3 &=& 2d_1^3-3d_1d_2+d_3,\hfill\\
c_4 &=& 6d_1^4-12d_1^2d_2+3d_2^2+4d_1d_3-d_4,\hfill\\
c_5 &=& 24d_1^5-60d_1^3d_2+30d_1d_2^2+20d_1^2d_3 \hfill
\\ && -10d_2d_3-5d_1d_4+d_5,\hfill\\
c_6 &=& \begin{smallmatrix}120d_1^6-360d_1^4d_2+270d_1^2d_2^2+120d_1^3d_3-30d_2^3-120d_1d_2d_3 \\
-30d_1^2d_4+10d_3^2+15d_2d_4+6d_1d_5-d_6\;. \hfill\end{smallmatrix}
\end{matrix}$
\end{tabular}
\end{example}
More generally, if $\alpha = a_{1}^{t_{1}}\cdots a_{\ell}^{t_{\ell}}$ is a partition, let $|\alpha|  = \sum_{i}t_{i}$ denote the number of parts of the partition and for any set of variables $\bf{x} = (x_{1},\dots,x_{n})$, let $x^{\alpha}:= x_{a_{1}}^{t_{1}} \cdots x_{a_{\ell}}^{t_{\ell}}$. 
Here is the relation between symmetrized principal minors and cycle-sums.
\begin{prop} \label{thm:symm-d_s=c_s}
Fix $n\ge 0$.
Consider rings $R_c = \CC\left[c_0,\dots,c_{n}\right]$ and $R_{d} = \CC\left[d_0,\dots,d_{n}\right]$. 
We have a (lower triangular) non-linear isomorphism of rings $R_{d} \to R_c$ given by:
\begin{equation}\label{symm-d_S=c_S}
d_{s}= \sum_{\alpha \vdash s} \; (-1)^{s-|\alpha|} \; p_{\alpha } \;  c^{\alpha},
\end{equation}
and a (lower triangular) non-linear isomorphism of rings $R_c \to R_{d}$ given by:
\begin{equation}\label{symm-c_S=d_S}
c_{s}= \sum_{\alpha \vdash s} \; (-1)^{s-|\alpha|} \; (|\alpha| -1)!\; p_{\alpha } \; d^{\alpha},
\end{equation}
where  
\[
p_{\alpha} = \frac{s!}{1!^{m_{1}} m_{1}! 2!^{m_{2}} m_{2}! \cdots s!^{m_{s}} m_{s}!}.
\]
is the number of set-partitions of $[s]$ with  $type(\alpha) = (m_{1},\ldots,m_{s})$.
\end{prop}

\begin{proof}
We simply combine the symmetrized terms in \eqref{d_S=c_S} and \eqref{c_S=d_S} to get \eqref{symm-d_S=c_S} and \eqref{symm-c_S=d_S}. The formula for $p_{\alpha}$ is \cite[Eq.3.37]{Stanley}.
\end{proof}

\begin{example}\label{ex-cd-4}
From Example \ref{ex-cd-[4]}, we immediately have:
\begin{eqnarray*}
 d_4 &=& -c_4+4c_3c_1+3c_2^2-6c_2c_1^2+c_1^4,
 \\
 c_4 &=& -d_4+4d_3d_1+3d_2^2-12d_2d_1^2+6d_1^4.
\end{eqnarray*}
\end{example}

\begin{example}
To express $d_6$ in terms of $c$'s, we compute $p_{\beta}$ and $\#\beta$ for all $\beta\vdash 6$:
\[{\small 
\begin{array}{|c|c|c|c|c|c|c|c|c|c|c|c|}
\hline
type(\beta)&6^{1}	&5^{1}1^{1}	&4^{1}2^{1}	&4^{1}1^{2}	&3^{2}	&3^{1}2^{1}1^{1} 	&3^{1}1^{3}	& 2^{3}	&2^{2}1^{2}	&2^{1}1^{4}	&1^{6}
\\ \hline
p_{\beta}	&1		&6			&15			&15			&10		&60				&20			&15		&45			&15			&1
\\ \hline
\#(\beta)	&1		&2			&2			&3			&2		&3				&4			&3		&4			&5			&6
\\ \hline

\end{array}}
\]
Therefore,  
\[d_{6}=
-c_6
+6c_1c_5
+15c_2c_4
-15c_1^2c_4
+10c_3^2
-60c_1c_2c_3 
+20c_1^3c_3
-15c_2^3
+45c_1^2c_2^2
-15c_1^4c_2
+c_1^6,
\]
and
\[c_{6}=
-d_6
+6d_1d_5
+15d_2d_4
-30d_1^2d_4
+30d_3^2
-120d_1d_2d_3 
+120d_1^3d_3
-30d_2^3
+270d_1^2d_2^2
-360d_1^4d_2
+120d_1^6.
\]
\end{example}

\section{Matrices with symmetrized cycle-sums}\label{sec:Matrices}
\subsection{Group actions preserving the SCS property}
In general, there are some group actions that preserve the symmetrized principal minor / cycle-sum property, which we call {\bf the SCS property}, and we call the set of all matrices with the SCS property SCS matrices.

Denote the following groups  in $\CC^{n\times n}$:
\begin{itemize}
\item ${\mathcal S}_{n}$:   the group of all $n\times n$ permutation matrices.  An element of ${\mathcal S}_n$ has the form
\begin{equation}
P_{\sigma}=[e_{\sigma(1)},e_{\sigma(2)},\cdots,e_{\sigma(n)}]=[e_{i_1},e_{i_2},\cdots,e_{i_n}],
\label{permutation-matrix}
\end{equation}
where $\sigma=\mtx{1,&2,&\cdots,&n\\i_1,&i_2,&\cdots,&i_n}$ is a permutation, and $e$'s are the standard basis of $\CC^{n}$. 

\item ${\mathcal D}_{n}$: the group of all $n\times n$ nonsingular diagonal matrices.  An element of ${\mathcal D}_n$ has the form
${\rm diag}(d_1,d_2,\cdots,d_n)$ with each $d_i\in \CC^*$. 

 \item
 ${\mathcal D}_{n}^{\pm}$: the group of all $n\times n$ diagonal matrices with $\pm 1$ as diagonal entries.  

\end{itemize}
Let us call  ${\mathcal S}_{n}\ltimes {\mathcal D}_{n}$  {\bf  the scalar permutation group}, and ${\mathcal S}_{n}\ltimes {\mathcal D}_{n}^{\pm}$   {\bf the sign permutation group}.
The following is straightforward to verify.

\begin{prop}\label{prop:symmetry}
Let $A \in \CC^{n\times n}$ be a matrix with the SCS property.
\begin{enumerate}
 \item {\bf Diagonal Modification}: The matrix $A-\lambda I_n$ has the SCS property; 
 \[c_1(A-\lambda I_n)=c_1(A)-\lambda; \quad \text{ and } \quad
c_k(A-\lambda I_n)=c_k(A) \;\;\text{ for all }\;\; k\ge 2.
\] 
For principal minors we have (set $d_{0}=1$)
 \[
 d_{k}(A-\lambda I) = \sum_{i=0}^{k} \binom{k}{i}(- \lambda)^{i}d_{k-i}(A)
. \]
  \item {\bf Homogeneity}: The scalar multiple $\lambda A$ 
  still has the SCS property;
\[c_k(\lambda A)=\lambda^k c_k(A) \quad \text{ and } \quad  d_{k}(\lambda A) = \lambda^{k} d_{k}(A)\;\; \text{ for all } \;\; k\ge 1.\]
\item{\bf Scalar-permutation Similarity}: The group action (by conjugation) of ${\mathcal S}_{n}\ltimes {\mathcal D}_{n}$  on the set of $n\times n$   SCS matrices
preserves all cycle-sums and all principal minors.   
\end{enumerate}
Moreover, operation 1 and conjugation by the  subgroup ${\mathcal S}_{n}\ltimes {\mathcal D}_{n}^{\pm}$   of ${\mathcal S}_{n}\ltimes {\mathcal D}_{n}$ 
preserve the set of $n\times n$ SCS symmetric (resp. skew-symmetric) matrices.
\end{prop}

When $c_2(A)\ne 0$, we can apply a {\bf normalization process} to $A$ as follows: define
\begin{equation}\label{normalization} 
A' :=\frac{1}{\sqrt{-c_2(A)}}\left(A-c_1(A)I_n\right),\qquad {\mathcal N}(A):=DA' D^{-1},
\end{equation}
where the diagonal entries of $D$ are given by: 
$$
d_{1,1}:=1;\qquad d_{k,k}:=\prod_{i=2}^{k} a'_{i-1,i}\quad\text{for \ } i=2,3,\cdots,n.
$$
Then $T:={\mathcal N}(A)\in \CC^{n\times n}$ has the SCS property,  $c_1(T)=0$, $c_2(T)=-1$; moreover, the diagonal entries $t_{i,i}=0$, the $+1$ diagonal entries $t_{i,i+1}=1$, and the $-1$ diagonal entries $t_{i+1,i}=-1$, for all appropriate indices.
The normalization process significantly simplifies the symbolic computations of cycle-sums and determinants (in \texttt{Macaulay2}, for instance) because it significantly reduces the number of parameters needed to express these quantities. 
 
\subsection{Symmetric SCS matrices}

\begin{theorem}\label{thm:SCS-symmetric}
Suppose $n\geq 2$ and $A \in S^{2}\CC^{n}$ has symmetrized $c_1$, $c_2$, and $c_3$ values.  Then $A$ has the  SCS property, and
\[
A =  c_{1}(A)I_{n} \pm \sqrt{c_{2}(A)} D (\mathbbm{1}_{n} - I_{n}) D^{-1}
,\]
where   $D\in {\mathcal D}_{n}^{\pm}$, and $\mathbbm{1}_{n}$ denotes the $n\times n$ all-ones matrix. In particular, $c_{k}(A) = \lambda^{k} c_{k}(\mathbbm{1}_{n})=\lambda^{k} (k-1)!$ for a fixed $\lambda \in\{\sqrt{c_{2}(A)},-\sqrt{c_{2}(A)}\}$ and $k\geq 2$.
\end{theorem}

\begin{proof}First apply diagonal modification to delete the diagonal of $A$, which does not change the rest of the cycle-sums. 
Up to re-scaling we may assume that  $c_{2}=1$ so that all off-diagonal entries must be $\pm 1$,   and $c_{3} = \pm 2$. 
We may further assume that $a_{i,i+1}=1$ for $i\in[n-1]$ by an appropriate ${\mathcal D}_{n}^{\pm}$-conjugation.

If $ c_{3}=2$, then  $a_{1,2}a_{2,3}a_{3,1} = 1$ so that $a_{1,3}=a_{3,1}=1$, and similarly $C_{\{i,i+1,j\}}(A) = 2$ implies that  $a_{i,j} =1$ for all $i,j$. 

If $c_{3}=-2$, then for any $i\in [n]$ and $i+2\le j\le n$, $C_{\{i,i+1,j\}}=2a_{i,i+1}a_{i+1,j}a_{i,j}=c_3=-2$, which implies that
$a_{i,j}=-a_{i+1,j}=(-1)^{2} a_{i+2,j}=\cdots=(-1)^{j-i-1}a_{j-1,j}=(-1)^{j-i-1}.$
Let $D:={\rm diag}(1, (-1)^1, (-1)^{2},\cdots,(-1)^{n-1})$.  Then $(-1)  D^{-1} A D = \mathbbm{1}_{n} - I_{n}$.  

Obviously, $A$ has the  SCS property.
\end{proof}

\begin{remark}
The condition that all the off-diagonal $2\times 2$ minors vanish is called exclusive-rank one (or E-rank one)
 in \cite{Oeding_tangential}. This ``off-diagonal rank'' was first studied in \cite{FiedlerMarkham} and is a special case of ``structure rank'' in \cite{FiedlerStructure, BrualdiMassey}. The fact that symmetric matrices with symmetrized cycle-sums can be written as the sum of a rank-one matrix and a diagonal matrix (and have E-rank $\leq 1$) can also be proved using Fiedler and Markham's main result in \cite{FiedlerMarkham}.

\end{remark}

\subsection{Skew-symmetric SCS matrices} 
Every skew-symmetric matrix has $c_{2k+1}=0$ for any odd integer $2k+1\in [n]$. 
Given $A\in\CC^{n\times n}$ and $S\in \mathcal{P} (n)$, let $A_{S}$ denote the principal submatrix of $A$ with row and column set $S$.

\begin{theorem}\label{thm:SCS-skew-symmetric}
Suppose $A \in \bw{2}\CC^{n}$ ($n\ge 4$) has symmetrized $c_2$ and $c_4$  values. Then $A$ has the  SCS property. Let $\mathbbm{1}_{n}^{\wedge}$ denote the $n\times n$ skew-symmetric matrix with 1's above the diagonal and $-1$'s below the diagonal. 
When $n\ne 4$,
\[
A = \lambda P\mathbbm{1}_{n}^{\wedge}P^{-1},
\qquad\text{for}\quad \lambda\in\CC,\ P\in {\mathcal S}_{n}\ltimes {\mathcal D}_{n}^{\pm}.
\]
When $n=4$, either $A = \lambda P\mathbbm{1}_{n}^{\wedge}P^{-1}$ for $\lambda\in\CC$ and $P\in {\mathcal S}_{n}\ltimes {\mathcal D}_{n}^{\pm}$ with $c_4(\mathbbm{1}_{n}^{\wedge})=2$, or
\[
A= 
\lambda P\mtx{
  0&  1&   1& 1\\
   -1&  0&  1&  -1\\
   -1&  -1&   0&  1\\
   -1&  1&  -1&  0\\
} P^{-1}
\qquad\text{for}\quad \lambda\in\CC,\ P\in {\mathcal S}_{4}\ltimes {\mathcal D}_{4}^{\pm},
\]
 with
$\ds
c_4\left(\mtx{
  0&  1&   1& 1\\
   -1&  0&  1&  -1\\
   -1&  -1&   0&  1\\
   -1&  1&  -1&  0\\
} \right)
=-6.$
\end{theorem}

\begin{proof}
Since $A\in\bw{2}\CC^{n}$   has symmetrized $c_2$ values, we have $A=\lambda T$ where 
$T$ is a skew-symmetric matrix that has $\pm 1$ as  all off-diagonal entries. Denote by $\bw{2}\{\pm 1\}^{n}$ the set of all such $T$'s.
Let us focus on $T$ and $\bw{2}\{\pm 1\}^{n}$.

We associate to every matrix  $T'\in\bw{2}\{\pm 1\}^{n}$   a multiset $\MS(T')=\{i_1, i_2,\cdots, i_n\}$, 
where $i_t$ is   the number of $-1$'s on the $t$-th row of $T'$.  For example, 
$\MS(\mathbbm{1}_{n}^{\wedge})=\{0,1,2,\cdots,n-1\}$.  The following observations are obvious:
\begin{itemize}
\item
If $\MS(T')=\{i_1, i_2,\cdots, i_n\}$, then $i_1+i_2+\cdots+i_n=\frac{n(n-1)}{2}$. 
\item
If $T', T''\in \bw{2}\{\pm 1\}^{n}$ are ${\mathcal S}_{n}$-conjugate, then $\MS(T')=\MS(T'')$. 
\item
If  $T'\in \bw{2}\{\pm 1\}^{n}$ is ${\mathcal S}_{n}\ltimes {\mathcal D}_{n}^{\pm}$-conjugate to $\mathbbm{1}_{n}^{\wedge}$ and 
$0\in\MS(T')$, then $\MS(T')=\{0,1,2,\cdots,n-1\}$.

\end{itemize}

When $n=3$, it is clear that $T=P\mathbbm{1}_{3}^{\wedge}P^{-1}$ for some $P\in {\mathcal S}_{3}\ltimes {\mathcal D}_{3}^{\pm}$. 

When $n=4$, $\MS(T)$ has 4 possibilities: $\{0,1,2,3\}$, $\{1,1,2,2\}$, $\{1,1,1,3\}$, and $\{0,2,2,2\}$. 
The matrices  in $\bw{2}\{\pm 1\}^{4}$ associated to $\{0,1,2,3\}$ and $\{1,1,2,2\}$ 
(resp. $\{1,1,1,3\}$ and $\{0,2,2,2\}$) are ${\mathcal S}_{4}\ltimes {\mathcal D}_{4}^{\pm}$-conjugate, with $c_4=2$ (resp. $c_4=-6$). 
For examples, let $D={\rm diag}(1,-1,1,1)$, then
\[
D
\mathbbm{1}_{n}^{\wedge}
D^{-1}=
\left(\begin{smallmatrix}
  0&  -1& 1& 1\\
1&  0&  -1&  -1\\
-1&  1&   0& 1\\
-1&   1&  -1&  0\\
\end{smallmatrix}\right),
\quad
D
\left(\begin{smallmatrix}
  0&  1&   1& 1\\
  -1&  0&  1&  -1\\
 -1&  -1&   0&  1\\
 -1&  1&  -1&  0\\
\end{smallmatrix}\right)
D^{-1}=
\left(\begin{smallmatrix}
  0&  -1&  1& 1\\
1&  0&  -1&  1\\
-1&  1&   0&  1\\
-1&  -1&  -1&  0\\
\end{smallmatrix}\right).
\]

Now suppose $n\ge 5$.  
We first prove that   $c_4(T)\equiv 2$.   Suppose on the contrary, $c_4(T)=-6$.  
Then up to ${\mathcal S}_{n}\ltimes {\mathcal D}_{n}^{\pm}$-conjugation, we may assume that   
\[
T_{\{1,2,3,4,5\}}=\mtx{0 &1 &1 &1 &1 \\
-1 &0 &1  &-1 &t_{2,5} \\
-1 &-1 &0 &1  &t_{3,5} \\
-1 &1 &-1 &0 &t_{4,5} \\
-1 &-t_{2,5} &-t_{3,5} &-t_{4,5} &0
}.
\]
By   $n=4$ case, we have $\MS(T_{\{1,2,3,5\}})=\{0,2,2,2\}$, so that  $t_{2,5}=-1$ and $t_{3,5}=1$. 
Similarly, $\MS(T_{\{1,3,4,5\}})=\{0,2,2,2\}$, so that $t_{3,5}=-1$.  This is a contradiction. Therefore, $c_4(T)\equiv 2$.

Finally, we prove by induction on $n\ge 5$ that every  matrix in $\bw{2}\{\pm 1\}^{n}$ that has symmetrized $c_2$ and $c_4$ values 
is ${\mathcal S}_{n}\ltimes {\mathcal D}_{n}^{\pm}$-conjugate to $\mathbbm{1}_{n}^{\wedge}$, which implies the SCS property of the matrix.  
We use $T$ as the example.

\begin{enumerate}
\item $n=5$:   Up to ${\mathcal S}_{5}\ltimes {\mathcal D}_{5}^{\pm}$-conjugation, we may assume that   
\[
T=\mtx{0 &1 &1 &1 &1\\-1 &0 &1 &1 &t_{2,5}\\-1 &-1 &0 &1 &t_{3,5}\\-1 &-1 &-1 &0 &t_{4,5}\\-1 &-t_{2,5} &-t_{3,5} &-t_{4,5} &0}.
\]
Then $\MS(T_{\{1,2,3,5\}})=\MS(T_{\{1,3,4,5\}})=\{0,1,2,3\}$. 
The possible cases are:
\begin{enumerate}
\item  $t_{2,5}=t_{3,5}=t_{4,5}=1$;
\item $t_{2,5}=t_{3,5}=1$, $t_{4,5}=-1$;
\item $t_{2,5}=1$, $t_{3,5}=t_{4,5}=-1$.
\end{enumerate}
All of them are ${\mathcal S}_{5}\ltimes {\mathcal D}_{5}^{\pm}$-conjugate to $\mathbbm{1}_{5}^{\wedge}$.  So $n=5$ is proved. 

\item $n=N$: Suppose the claim is true for any $n$ with $5\le n<N$.   Then for $T=(t_{i,j})_{n\times n}\in \bw{2}\{\pm 1\}^{N}$, up to ${\mathcal S}_{N}\ltimes {\mathcal D}_{N}^{\pm}$-conjugation,
we may assume that $T_{\{1,2,\cdots,N-1\}}=\mathbbm{1}_{N-1}^{\wedge}$, and  $t_{1,N}=1$. 
By assumption, 
$T_{\{1,2,\cdots,N-2,N\}}$ is  ${\mathcal S}_{N-1}\ltimes {\mathcal D}_{N-1}^{\pm}$-conjugate to $\mathbbm{1}_{N-1}^{\wedge}$. 
Moreover, $0\in\MS(T_{\{1,2,\cdots,N-2,N\}})$. So $\MS(T_{\{1,2,\cdots,N-2,N\}})=\{0,1,2,\cdots,N-2\}$. 
Thus it is impossible to have $t_{i,N}=-1$ and $t_{i+1,N}=1$ for any $i=2,3,\cdots,N-3$. Similarly, it is impossible to have $t_{N-2,N}=-1$ and $t_{N-1,N}=1$. 
Therefore, the possible cases for $T$ are:
\[
t_{1,N}=\cdots=t_{i,N}=1,\quad
t_{i+1,N}=\cdots=-1,\quad
t_{N,N}=0,\quad
\text{for some $i\in [N-1]$.}
\]
All of them are ${\mathcal S}_{N}\ltimes {\mathcal D}_{N}^{\pm}$-conjugate to $\mathbbm{1}_{N}^{\wedge}$.  
\end{enumerate}
Therefore, the claim holds and  the proof is done.
\qedhere
\end{proof}

\subsection{Arbitrary square SCS matrices}\label{sec:nonsym}
We discuss arbitrary square matrices $A$ in $\CC^{n\times n}$ with SCS property in this section.  
For simplicity, we assume that $c_1=0$. We handle these matrices in 3 cases: when $c_{2}=0$, when $c_{2}\neq 0$ and $c_{3}=0$, and when $c_{2}\neq 0$ and $c_{3}\neq0$.
When $c_2\ne 0$, we also assume that $A$ is  normalized (see \eqref{normalization}), 
so that $A={\mathcal N}(A)$ and $c_2=-1$.  

\subsubsection{Case $c_2=0$}
This case includes the following two examples:
\begin{enumerate}
\item
Any strictly upper triangular matrix $A$ satisfies that
$
c_1=c_2=\cdots=c_{n}=0.
$
\item
The permutation matrix 
$P_{\mtx{1, &2, &\cdots, &n-1, &n\\2, &3, &\cdots, &n, &1}}$ defined in \eqref{permutation-matrix} has  
\[
c_1=c_2=\cdots=c_{n-1}=0,\qquad c_n=1.
\]
\end{enumerate}

The following theorem embraces both examples: 

\begin{theorem}\label{thm:c_2=0}
Suppose $n\ge 2$ and $A\in\CC^{n\times n}$ is a SCS matrix with  $c_1=c_2 =0$.
Then $A$ belongs to one of the following situations:
\begin{enumerate}
\item 
$A$ is ${\mathcal S}_n$-conjugate  to a strictly upper triangular matrix, where
$$c_1=c_2=\cdots=c_n=0.$$

\item
$A$ is an element of ${\mathcal S}_n \ltimes {\mathcal D}_n$ with  the ${\mathcal S}_n$ component of order $n$, where $n\ge 3$ and 
$$c_1=c_2=\cdots=c_{n-1}=0,\qquad c_n\ne 0.$$

\end{enumerate}

\end{theorem}

\begin{proof}
\begin{enumerate}
\item
We use induction to prove the claim: if a SCS matrix $A\in\CC^{n\times n}$ satisfies that $c_1=c_2=\cdots=c_n=0$, then $A$ is   ${\mathcal S}_n$-conjugate to a strictly upper triangular matrix.  

The case $n=2$ is immediate.  
Suppose the claim holds for any integer $m$ with $2\le m<n$. Now let $A=(a_{i,j})_{n\times n} \in\CC^{n\times n}$ be  a SCS matrix with $c_1=c_2=\cdots=c_n=0$.
By induction hypothesis, up to ${\mathcal S}_n$-conjugation, we may assume that the principal submatrix $A_{\{1,2,\cdots,n-1\}}$  is strictly upper triangular.  
 
\begin{enumerate}
\item
If  $a_{n,1}=0$, then the first column of $A$ is  zero. By induction hypothesis, there exists a ${\mathcal S}_n$-conjugation that permutes the last $(n-1)$ rows and columns of $A$ respectively, such that the resulting matrix $A'$ has a  strictly upper triangular principal submatrix $A'_{\{2,3,\cdots,n\}}$. Since the first column of $A'$ is still  zero,  $A'$ is strictly upper triangular.  The claim is proved.

\item
If  $a_{n,1}\ne 0$,  
we first show that there exists a zero row in $A$.  Suppose on the contrary, every row of $A$ is nonzero.  
From the first row, pick   $i_2>1$ such that  $a_{1,i_2}\ne 0$.  If $i_2<n$, pick $i_3>i_2$ such that  $a_{i_2,i_3}\ne 0$.  Repeat the process until we reach $i_{\ell}=n$.  
Then the $\ell$-cycle-sum
$C_{\{1,i_2,\cdots,i_{\ell-1},n\}} (A)=a_{1,i_2}a_{i_2,i_3}\cdots a_{i_{\ell-1},n}a_{n,1}\ne 0$, which contradicts to the assumption $c_{\ell}=0$.  So $A$ has a zero row.  
Then $A$ is ${\mathcal S}_n$-conjugate to a matrix $A'$ with a zero $n$-th row. 
By induction hypothesis,  there exists an ${\mathcal S}_n$-conjugation    that permutes the first $(n-1)$ rows and columns of $A'$ respectively, 
and the resulting matrix $A''$ has strictly upper triangular $A''_{\{1,2,\cdots,n-1\}}$.  Then $A''$ is strictly upper triangular,
 and the claim is proved.
\end{enumerate}

Overall, the claim holds for all $n$.

\item Now we prove the following claim: if a SCS matrix $A=[a_{i,j}]_{n\times n}\in\CC^{n\times n}$ has $c_1=c_2=\cdots =c_{k-1}=0$  but $c_k\ne 0$ for certain $k$ with $3\le k\le n$, then $k=n$, and $A\in {\mathcal S}_n \ltimes {\mathcal D}_n$ has the ${\mathcal S}_n$ component of order $n$.  This will complete the proof of the whole theorem. 

Since $c_1=c_2=\cdots =c_{k-1}=0$, up to ${\mathcal S}_n$-conjugation, we may assume that $A_{\{1,\cdots,k-1\}}$ is strictly upper triangular. 
 There exists a ${\mathcal S}_n$-conjugation on $A$  that permutes the rows and columns in $\{2,3,\cdots,k\}$, such that the resulting matrix $A'=[a'_{i,j}]_{n\times n}$ has
a strictly upper triangular  $A'_{\{2,3,\cdots,k\}}$.
 Since the first column of $A_{\{1,2,\cdots,k\}}$ has at most one nonzero entry, so does the first column of $A'_{\{1,2,\cdots,k\}}$. 
  Then  
 $$0\ne c_k=C_{\{1,2,\cdots,k\}}(A')=a'_{1,2} a'_{2,3}\cdots a'_{k-1,k} a'_{k,1}.$$ 
So $a'_{k,1}\ne 0$, and it is the only nonzero entry in the first column and in the lower triangular part of $A'_{\{1,2,\cdots,k\}}$.
We declare that $a'_{1,2}, a'_{2,3},\cdots, a'_{k,1}$ are the only nonzero entries in  $A'_{\{1,2,\cdots,k\}}$; otherwise, $a'_{i,j}\ne 0$ for some 
$1\le i<i+1<j\le k$, and the cycle-sum $C_{\{1,2,\cdots,i,j,j+1,\cdots,k\}}(A')\ne 0$, which contradicts to  $c_1=c_2=\cdots =c_{k-1}=0$.

If $k=n$, then   both $A$ and $A'$ are elements of ${\mathcal S}_n \ltimes {\mathcal D}_n$ with the ${\mathcal S}_n$ component of order $n$. So the claim holds. 

It remains to prove that $k<n$ is impossible.  Otherwise, $3\le k<n$.  Then
\[
A'_{\{1,2,\cdots,k+1\}}=
\mtx{
 &a'_{1,2} & & &a'_{1,k+1} \\ & &\ddots & &\vdots \\ & & &a'_{k-1,k} &a'_{k-1,k+1} \\
a'_{k,1} &0 &\ldots &0 &a'_{k,k+1} \\ a'_{k+1,1} &a'_{k+1,2} &\ldots & a'_{k+1,k} &0
}.
\]
By $c_{k}=C_{\{2,3,\cdots,k+1\}}(A')\ne 0$, we have $a'_{k,k+1}\ne 0$ and $a'_{k+1,2}\ne 0$. 
By $c_{k}=C_{\{1,3,4,\cdots,k+1\}}(A')\ne 0$, we have $a'_{k+1,3}\ne 0$.  Then
\[
c_{k-1}=C_{\{3,4,\cdots,k+1\}}(A')=a'_{3,4}\cdots a'_{k,k+1}a'_{k+1,3}\ne 0,
\] 
contradicting to the assumption $c_{k-1}=0$.  Therefore, $k=n$.
\qedhere
\end{enumerate}
\end{proof}

\subsubsection{Case $c_2\ne 0$, $c_3= 0$}\label{sec:c2c3ne0} 

A typical family of SCS matrices with $c_2\ne 0$ and $c_3=0$ can be found in skew-symmetric matrices. See Theorem \ref{thm:SCS-skew-symmetric}. 
Indeed,  any matrix  with $c_1=0$, $c_2\ne 0$, and $c_3=0$  is diagonal conjugate to a skew-symmetric one.

\begin{theorem}\label{thm:c_3=0}
Suppose $A\in\CC^{n\times n}$ has symmetrized $c_k$ values for $k=1,2,3$, with $c_1=0$, $c_2\ne 0$ and $c_3=0$.  
Then there exists   $D\in{\mathcal D}_n$ such that
\begin{equation}
A=D^{-1}(\sqrt{-c_2}T)D,
\label{normalize-c_3=0}
\end{equation}
where $T$ is a normalized skew-symmetric matrix, i.e., all off-diagonal entry values of $T$ are $\pm 1$, and  $t_{1,2}=t_{2,3}=\cdots=t_{n-1,n}=1$.
In particular, $A$ has symmetrized $c_{2k+1}$ values with $c_{2k+1}=0$ for all $k\ge 1$.
\end{theorem}

Together with Theorem \ref{thm:SCS-skew-symmetric}, we get the following result about SCS matrices:

\begin{cor}\label{cor:skew-param}
Suppose $A\in\CC^{n\times n}$ is a SCS matrix with $c_1=0$, $c_2\ne 0$ and $c_3=0$.  
Then $A$ is ${\mathcal S}_n\ltimes {\mathcal D}_n$-conjugate to $\mathbbm{1}_{n}^{\wedge}$ (defined in  Theorem \ref{thm:SCS-skew-symmetric}) 
or $\mtx{
  0&  1&   1& 1\\
   -1&  0&  1&  -1\\
   -1&  -1&   0&  1\\
   -1&  1&  -1&  0
}
$
(for $n=4$ only).
\end{cor}

\begin{proof}[Proof of Theorem \ref{thm:c_3=0}]
By the normalization process \eqref{normalization}, there is $D\in {\mathcal D}_n$ such that 
$A=\sqrt{-c_2}D^{-1} {\mathcal N}(A)  D$. Then $T:={\mathcal N}(A)$ is a SCS matrix with $c_1(T)=0$, $c_2(T)=-1$, $c_3(T)=0$, 
 $t_{1,2}=t_{2,3}=\cdots=t_{n-1,n}=1$, and $t_{2,1}=t_{3,2}=\cdots=t_{n,n-1}=-1$.  It remains to show that
 $t_{i,i+k}\in\{1,-1\}$ for any $1\le i< i+k\le n$, thereof $t_{i+k,i}=c_2(T)/t_{i,i+k}=-t_{i,i+k}$ and  $T$ is skew-symmetric. 
Let us make induction on $k$. $k=1$ is obvious.
Suppose $t_{i,i+k}\in\{1,-1\}$ for all $k< m$ ($m\ge 2$) and all index pairs $(i,i+k)$ of $T$.   Then for any index pair $(j,j+m)$ of $T$,
$$
0=c_3(T)=t_{j,j+1}t_{j+1,j+m}t_{j+m,j}+t_{j,j+m}t_{j+m,j+1}t_{j+1,j}=t_{j+m,j}+t_{j,j+m}.
$$
Moreover, $1=c_2(T)=t_{j+m,j}t_{j,j+m}$.  Therefore, $t_{j,j+m}\in\{1,-1\}$.  The proof is done.
\end{proof}

\subsubsection{Case $c_2\ne 0$, $c_3\ne 0$}  

We consider the following Toeplitz matrix for  any $x\in\CC^{*}$:
\begin{equation}\label{SCS-canonical}
T_n(x):=
\mtx{
      0 &1 &x &x^2 &\cdots &x^{n-2}
\\ -1 &0 &1 &x &\cdots &x^{n-3}
\\ -\frac{1}{x} &-1 &0 &1 &\cdots &x^{n-4}
\\ -\frac{1}{x^{2}} &-\frac{1}{x} &-1 &0 &\cdots &x^{n-5}
\\ \vdots &\vdots &\vdots &\vdots &\ddots &\vdots
\\ -\frac{1}{x^{n-2}} &-\frac{1}{x^{n-3}} &-\frac{1}{x^{n-4}} &-\frac{1}{x^{n-5}} &\cdots &0
},
\end{equation}
where the  $(i,j)$ entry of $T_n(x)$ is exactly $\sgn(j-i)\cdot x^{j-i-\sgn(j-i)}$. 

For a permutation $w \in \mathfrak{S}_{n}$ write the word $w = w_{1}w_{2}\cdots w_{n}$ if  as a bijection on $[n]$ we have $w(i) = w_{i}$. 
A descent in $w$ is a position $i$ such that $w_{i}>w_{i+1}$. 
Let $des(w)$ denote the number of its descents.  
The Euler number $E(k,i)$ is the number of permutations $w\in \mathfrak{S}_{k}$ with exactly $i-1$ descents. (See \cite[Ch.~1.4]{Stanley}).
Given the description of the cycle-sums of the special Toeplitz matrix $T_{n}(x)$, the following is straightforward to verify.

\begin{theorem}\label{thm:eulerian}
The matrix $T_n(x)$ satisfies the SCS property. In particular, for $k\geq 2$ the cycle-sums of $T_{n}(x)$ are the (re-scaled, signed) Eulerian polynomials
\begin{eqnarray*}
c_{k}(T_{n}(x)) & =& \frac{1}{x^{k}} \sum_{w \in \mathfrak{S}_{k-1}} (-x^2)^{des(w)+1}
\\
& = & x^{-k} \sum_{i=1}^{k-1} E(k-1,i)\cdot (-x^{2})^{i}
.
\end{eqnarray*}
\end{theorem}

\begin{proof}
Let $T_{n}(x)=(t_{ij})_{n\times n}$. Every summand in a $k$-cycle-sum $C_{I}(T_{n}(x))$ has the form
\begin{eqnarray*}
&& t_{i_1,i_2}t_{i_2,i_3}\cdots  t_{i_{k},i_{1}}
\\
&=& 
\sgn(i_2-i_1) x^{i_2-i_1-\sgn(i_2-i_1)}
\sgn(i_3-i_2) x^{i_3-i_2-\sgn(i_3-i_2)}
\cdots
\sgn(i_1-i_k) x^{i_1-i_k-\sgn(i_1-i_k)}
\\
&=&
\left(\sgn(i_2-i_1)\sgn(i_3-i_2)\cdots\sgn(i_1-i_k) \right)
x^{-\sgn(i_2-i_1)-\sgn(i_3-i_2)-\cdots -\sgn(i_1-i_k)},
\end{eqnarray*}
which solely  depends on the relative order of the indices in the circle $(i_1,i_2,\cdots,i_k)$. 
In particular, we can express each term in terms of descents so that the formula for the cycle-sums follows.
Suppose $i_1=\min\{i_1,i_2,\cdots,i_k\}$.  The circle $(i_1,i_2,\cdots,i_k)$ corresponds to the permutation $w \in \mathfrak{S}_{k-1}$ with the same relative order as $i_2i_3\cdots i_k$, such that
\[
 t_{i_1,i_2}t_{i_2,i_3}\cdots  t_{i_{k},i_{1}} = (-1)^{des(w)+1}\cdot x^{2(des(w)+1)-k}.
\]
Therefore, $C_{I}(T_{n}(x))=C_{J}(T_{n}(x))$ for any $I, J\subset [n]$ and $|I|=|J|$. \end{proof}

The matrix $\mathbbm{1}_{n}^{\wedge}$ 
in Theorem \ref{thm:SCS-skew-symmetric} and Corollary \ref{cor:skew-param} 
is exactly $T_n(1)$. 
In fact, it turns out that every general SCS matrix with $c_{1}=0$ is  a ${\mathcal S}_{n}\ltimes {\mathcal D}_{n}$-conjugate of the Toeplitz matrix $\lambda T_{n}(x)$ for some $\lambda\in\CC$.

\begin{theorem}\label{thm:c_2!=0,c_3!=0}
Suppose $A\in\CC^{n\times n}$ has  symmetrized $c_k$ values for $k=1,2,3$, with  $c_1=0$, $c_2\ne 0$, and $c_3\ne 0$. Then $A$ is a SCS matrix;
and $A$ is ${\mathcal S}_n\ltimes {\mathcal D}_n$-conjugate to $\lambda T_{n}(x)$,
where    $\lambda^2=-c_2$ and $\lambda^3(x-\frac{1}{x})=c_3$.
\end{theorem}

\begin{proof}
After re-homogenizing, we can assume $c_{2}=-1$  and $\lambda=1$. 
By induction we may assume that the statement is true for a fixed $n$ with $n\geq 3$. Let $B$ be the matrix with $T_{n}(x)$ in the upper-left corner, padded by the column $(y_{1},\ldots,y_{n-1},y_{n},0)^{t}$, and the row $(z_{1},\ldots,z_{n})$.  We may assume, since $y_{i}z_{i}=c_{2}=-1$, that $y_{i}\neq 0$ and $z_{i}=-1/y_{i}$. Further, by conjugating by $diag(1,\ldots,1,y_{n})$ and renaming each $y_{i}/y_{n}$ by $y_{i}$, that the matrix $B$ is the matrix padded by the column $(y_{1},\ldots,y_{n-1},1,0)^{t}$, and the row $(-1/y_{1},\ldots,-1/y_{n-1},-1)$.

Now setting all instances of $c_{3}(B)$ to be equal, we have equations of the following form:
\[
c_{3} = x-\frac{1}{x} =  \frac{y_{i}}{x^{n-1-i}} - \frac{x^{n-1-i} }{y_{i}} .
\]
Solving these equations, we find that $y_{i} = -x^{n-i-2}$ or $y_{i} = x^{n-i}$.  In particular, every entry in the padded row / column must be a power of $x$.

We also have equations of the form
\[
c_{3} =   \frac{y_{i}}{y_{i+1}} - \frac{y_{i+1} }{y_{i}}.
\]
Adjacent entries in the padded row/column must either increase or decrease by one power of $x$, and if they increase as $i$ increases, the sign changes.

If all $y_{i}=x^{n-i}$, then $B=T_{n+1}(x)$.  Otherwise, we find the greatest $t$ such that $y_{t+1}=x^{n-t-1}$ but $y_{t}=-x^{n-t-2}$.
Then  $y_{t-1}=-x^{n-t-1}$, $y_{t-2}=-x^{n-t}$, and so on. 

Finally, all such matrices $B$ are conjugate to $T_{n+1}(x)$:  let 
\[
D={\rm diag}(\underbrace{1,\cdots,1}_{t},-\frac{1}{x^{n-t-2}},\underbrace{x,\cdots,x}_{n-t}),\qquad
\sigma=(t+1, t+2,\cdots,n+1)\in \mathfrak{S}_{n+1},
\]  
then $(P_{\sigma}D)T_{n+1}(x)(P_{\sigma}D)^{-1}=B$.
\end{proof}


\section{Polynomial relations among symmetrized cycle-sums}\label{sec:polynomials}

In this section we analyze the ideals of the varieties of symmetrized cycle-sums and symmetrized principal minors of symmetric, skew-symmetric, and general $n\times n$ matrices. 

\subsection{The case of symmetric SCS matrices}\label{sec:sym relations}

The following is straightforward to verify, and implies, in particular, that the variety of symmetrized principal minors / cycle-sums of symmetric matrices is toric because it provides a monomial parametrization (see, for instance, \cite[Ch.~7.1]{miller2005combinatorial}).
\begin{lemma}\label{symParam}
Suppose  $A =  a\cdot I_{n}  + b\cdot \mathbbm{1}_{n}$
with $a,b \in \CC$. 
Then $d_{1}= c_{1} = a+b$ and 
for all $S\subset [n]$ with $|S|\geq 2$
\[
D_{S} = d_{|S|} =a^{|S|-1}(a+|S|b) 
,\quad\quad
\text{and }\quad \quad
 \quad C_{S} = c_{|S|} = (|S|-1)! b^{|S|}.\]
\end{lemma}

\begin{theorem}\label{thm:sym}
Let $\J_{n}^{\circ}$ denote the ideal of the variety $Z_{n}^{\circ} \cap S^{n}\CC^{2} \cap U_{c_{0}=1}$.
 If $n=3$ then $\J^{\circ}_{n}$ is prime, and generated by a single equation, 
 \[
 \J^{\circ}_{3}=
 \langle 4 c_{2}^{3}-c_{3}^{2} \rangle 
 .\]
For $n\geq 4$ $\J^{\circ}_{n}$ is
the prime ideal generated by the following $n-2$ binomials:
\[
\left\{4 c_{2}^{3}-c_{3}^{2}\right\} \cup
\left\{ (s-1)!c_{2}c_{s-2} - (s-3)!c_{s} \;\mid\; 4\leq s\leq n
\right\}.
\]
\end{theorem}
\begin{proof} The case $n=3$ can be verified immediately by direct computation.
By Theorem~\ref{thm:SCS-symmetric} the pull-back of the ``symmetrized principal minors'' condition to the space of matrices cuts out (as a set) a space of matrices of the form $A = a\cdot I_{n} + b\cdot \mathbbm{1}_{n}$, with $\mathbbm{1}_{n}$ the $n\times n$ all-ones-matrix. 
There are two cases to consider, depending on whether we invert $b$ or not.

First when $b=0$, in which case Lemma~\ref{symParam} implies that $c_{s}=0$ for $s\geq 2$. Thus as a set, we have identified a subscheme of $\V(\J^{\circ}_{n})$ supported on the line defined by $\langle c_{s} \mid 2\leq s \leq n\rangle$.  This ideal, however, is generally not radical, and (loosely) reflects the different orders of vanishing of the cycle-sums.

Now assume  $b\neq 0$. Set $J= \langle \left\{4 c_{2}^{3}-c_{3}^{2}\right\} \cup
\left\{ (s-1)!c_{2}c_{s-2} - (s-3)!c_{s} \;\mid\; 4\leq s\leq n
\right\} \rangle.$
It is straightforward to check that the zeroset of $J$ contains the image of the paramatrization  given in Lemma~\ref{symParam}. On the other hand, $J$ is the ideal of the graph in $S^{n}\CC^{2}\cap U_{c_{0}=1}$ of the curve $\langle 4c_{2}^{3}-c_{3}^{2}\rangle  \subset \CC\{c_{2},c_{3}\}$, given by the monomial functions $\{ c_{s} = \frac{(s-1)!}{(s-3)!}c_{2}c_{s-2} \mid 4\leq s\leq n\}$, where for $s\geq 5$ we recursively replace $c_{s-2}$ until we obtain a monomial in $c_{2}$ and $c_{3}$. Being the ideal of the graph of an irreducible curve, $J$ is prime. So the inclusion $J \subset I(Z_{n}^{\circ} \cap S^{n}\CC^{2}\cap U_{c_{0}=1})$, is an inclusion of prime ideals of the same dimension, so it must be an equality. Finally, the radical of the ideal obtained in the case $b=0$ contains $J$, and thus corresponds to an embedded line in $Z_{n}^{\circ} \cap S^{n}\CC^{2}\cap U_{c_{0}=1}$.
%
%
%
%
%
\end{proof}

\begin{example}\label{ex:J4}
If $n=4$, then elimination (in \texttt{Macaulay2}) reveals that
the symmetrized ideal of relations among cycle sums  is the intersection 
  \[
\langle3 c_{3}^{2}-2 c_{2} c_{4},6 c_{2}^{2}-c_{4} \rangle
\cap
\langle c_{4},c_{3}^{2},c_{2}^{2} c_{3},c_{2}^{3} \rangle
,\]
the first of which is the prime ideal of $\tau(\nu_{4}\PP^{1})\cap \{c_{0}=1\}$ and corresponds to $\J_{4}^{\circ}$ in Theorem~\ref{thm:sym}, and the second of which is a non-prime ideal supported on the vanishing set of the cycle-sums $\{c_{2},c_{3},c_{4}\}$. The radical of the second ideal evidently contains the first and (geometrically) corresponds to a line embedded in the scheme supported on $\tau(\nu_{4}\PP^{1})\cap \{c_{0}=1\}$.

To recover the relations amongst symmetrized principal minors a straightforward elimination calculation (again  in \texttt{Macaulay2}) produces the intersection
\[
\langle\begin{smallmatrix} 3 {d}_{2}^{2}-4 {d}_{1} {d}_{3}+{d}_{4},\\ 2 {d}_{1} {d}_{2} {d}_{3}-3 {d}_{1}^{2} {d}_{4}-3 {d}_{3}^{2}+4 {d}_{2} {d}_{4}\end{smallmatrix} \rangle
\cap
\left \langle 
\begin{smallmatrix}
2 {d}_{2} {d}_{3}^{2}-{d}_{2}^{2} {d}_{4}-4 {d}_{1} {d}_{3} {d}_{4}+3 {d}_{4}^{2},&{d}_{2}^{2} {d}_{3}-2 {d}_{1} {d}_{3}^{2}-2 {d}_{1} {d}_{2} {d}_{4}+3 {d}_{3} {d}_{4},\\
2 {d}_{1} {d}_{2} {d}_{3}+{d}_{1}^{2} {d}_{4}-{d}_{3}^{2}-2 {d}_{2} {d}_{4},&{d}_{2}^{3}+2 {d}_{1}^{2} {d}_{4}-3{d}_{2} {d}_{4},\\
{d}_{1} {d}_{2}^{2}+2 {d}_{1}^{2} {d}_{3}-2 {d}_{2} {d}_{3}-{d}_{1} {d}_{4},&2 {d}_{1}^{2} {d}_{2}-{d}_{2}^{2}-{d}_{4},{d}_{1}^{4}-{d}_{4}
      \end{smallmatrix}
\right \rangle
.\]
 The geometric structure of this decomposition is less evident in principal minor coordinates, but because the degree and number of variables are small we can still perform the computations. The first ideal is prime. The second ideal is not, but has radical $\langle {d}_{1}^{2}-{d}_{2},{d}_{1}^{3}-{d}_{3},{d}_{1}^{4}-{d}_{4} \rangle$, and one can check that the radical of the second ideal contains the first.  In general, the elimination calculation using symmetrized principal minors becomes difficult once $n\geq 5$.
\end{example}

The following characterizes the principal minors of the symmetric E-rank one matrices.
\begin{prop} \cite[Prop.~5.2]{Oeding_tangential} The image of the principal minor map of $n\times n$ symmetric matrices of E-rank one is the tangential variety of the Segre product of $n$ projective lines.
\end{prop}

A weaker version of the main theorem in \cite{OedingRaicu} (which was a conjecture of Landsberg and Weyman \cite{LanWey_tan}, and proved set-theoretically in \cite{Oeding_tangential}) is the following
\begin{theorem}
The ideal of the tangential variety $\tau(\Seg(\PP^{1}\times\dots \times \PP^{1}))$ is generated by the Landsberg-Weyman equations (a specific set of quadric, cubic, and quartic polynomials).
\end{theorem}

Suppose $A$ is an $n\times n$ symmetric matrix with generic entries. Let $\I_{n}$ denote the ideal of relations amongst the principal minors of $A$. 
The zero set of the symmetrization  $\I_{n} \cap S^{n}\CC^{2}$ is a subvariety of the tangential variety of the Segre product of $n$ projective lines subject to additional symmetry. In particular, it is the tangential variety of the degree $n$ rational normal curve (the degree $n$ Veronese embedding of the projective line). The minimal generators of this ideal (were likely known classically) are also determined as a special case of the main theorem in \cite{OedingRaicu}.  See \cite[Ch.~10]{DolgachevAG} for an in-depth investigation such classical varieties; the tangential surface of the rational normal curve is discussed in \cite[Ex.~10.4.14]{DolgachevAG}.

We end this section with the following geometric characterization of the previous result.
\begin{cor}\label{thm:sym_case}
The scheme of symmetrized principal minors (symmetrized cycle-sums) of symmetric matrices
\[
Z^{\circ} \cap \PP S^{n}\CC^{2} \cap U_{c_{0}=1}
\]
consists of an affine section of the tangential variety of the degree $n$ Veronese embedding of $\PP^{1}$ (the rational normal curve) $\tau(v_{n}(\PP^{1})) \cap U_{c_{0}=1}$ together with a high-degree scheme whose reduced structure is the line corresponding to the condition $c_{s}=0$ for $2\leq s \leq n$.
\end{cor}
\begin{proof} By Theorem~\ref{thm:SCS-symmetric} the underlying variety is the symmetrization of the variety of principal minors of ``rank-one plus diagonal'' symmetric matrices. The latter variety was already shown to be the tangential variety of the Segre product of $n$ copies of $\PP^{1}$, \cite[Prop.~5.2]{Oeding_tangential}. Symmetrizing the resulting variety gives the geometric result.  The embedded ideal is found by considering the case that the rank-one matrix is actually the zero matrix.
\end{proof}

\begin{remark}
Our procedure of working on an affine open set, pulling back the symmetrization condition to the space of matrices through the cycle-sum map, restricting the source, then looking at the relations among the coordinates of the restricted image, is a loss of ideal-theoretic information. Therefore, our results only hold on an open subset of the projective scheme.
\end{remark}

\subsection{The case of skew-symmetric SCS matrices}\label{sec:skew relations}
The general principal Pfaffian assignment problem is already solved. This is because the map that takes a skew-symmetric matrix to a vector of all of its principal Pfaffians actually defines the orthogonal Grassmannian, whose ideal is known to be generated by the analog of the Pl\"ucker relations.

It still would be interesting to understand what happens when we set principal Pfaffians of equal size to be equal, but in this article, we focus on principal minors and cycle-sums. It is well known that all odd principal minors of a skew-symmetric matrix are zero.  A similar result holds for cycle-sums. 
\begin{prop} Suppose $A\in \bw{2}\CC^{n}$. Then all  odd cycle-sums of $A$ are zero. 
\end{prop}\label{prop:odd}
\begin{proof}
Consider a cycle-sum
\[
C_{S} = \sum_{\{i_{1},i_{2},\dots,i_{|S|} \} = S } a_{i_{1},i_{2}}a_{i_{2},i_{3}}\cdots a_{i_{|S|-1},i_{|S|}}a_{i_{|S|},i_{1}}
.\]
Notice that if we reverse the direction of all the cycles on $S$,  $C_{S}$ remains unchanged. If $|S|$ is odd, then each term changes sign when the cycle is reversed (because $a_{i,j} = -a_{j,i}$), implying that $C_{S} = -C_{S}$, so $C_{S}=0$.
\end{proof}

Here is a characterization of SCS skew-symmetric matrices and their principal minors.

\begin{prop}
Suppose $A \in \bw{2} \CC^{n}$ is nonzero, has symmetrized cycle-sums and $n\neq 4$. Then there is $\lambda \in \CC$ so that $d_{k}(\lambda A) = d_{k}(\mathbbm{1}^{\wedge}) = 1$ for $k\geq 2$.
\end{prop}
\begin{proof}
By Theorem~\ref{thm:SCS-skew-symmetric} we have that a scalar multiple of $A$ is conjugate to $\mathbbm{1} ^{\wedge}$.  The fact that $d_{k}(\mathbbm{1}^{\wedge})=1$ for $k\geq 2$ follows by induction and using $(n-2)\times (n-2)$ Schur complements.
\end{proof}
The cycle-sums of skew-symmetric SCS matrices take a nice form.  In the case $n=4$ there are two possible values for $c_{4}$.

\begin{example}\label{ex:n4}
The following two matrices respectively have $c_{4}= 2b^{4}$ and $c_{4}=-6b^{4}$
 \[\begin{pmatrix}
 0&      {-b}&   b&      {-b}\\
 b&      0&      {-b}&     b\\
      {-b}&      b&      0&      b\\
      b&      {-b}&      {-b}&      0\\
      \end{pmatrix},\qquad
\begin{pmatrix}0&      b&      b&      b\\
      {-b}&      0&      b&      {-b}\\
      {-b}&      {-b}&      0&      b\\
      {-b}&      b&      {-b}&      0\\
      \end{pmatrix}.\]
      Notice that neither has off-diagonal rank one.
\end{example}

The following is an immediate consequence of Theorem~\ref{thm:eulerian}   and Theorem~\ref{thm:SCS-skew-symmetric}, and yields a parametrization of the  cycle-sums of skew-symmetric SCS matrices. 

\begin{lemma}\label{skewParam}
Consider  $ \mathbbm{1}^{\wedge}$, the canonical $n\times n$ skew-symmetric SCS matrix, and suppose $n\geq 5$.
Then $c_{s}(  \mathbbm{1}^{\wedge}) = (-1)^{s/2}E_{s-1}$, where $E_{n}$ is the Euler number. 
In particular, $c_{s}(\mathbbm{1}^{\wedge})$ has  the following exponential generating function:
\[
\sum_{s\geq 0} c_{s+1}(  \mathbbm{1}^{\wedge}) \frac{(-x)^{s}}{s!} = \tan(x)
.\]
\end{lemma}
\begin{proof} 
The odd-sized cycle-sums vanish for skew-symmetric matrices, so work with $S$ with even size.
Organize the computation of the cycle-sum $c_{S}(\mathbbm{1}^{\wedge})$ by picking a distinguished element $a\in S$, and summing over permutations of $S\setminus \{a\}$. Then each term indexed by an alternating permutation contributes $(-1)^{|S|/2}$ to the sum. So, up to sign, $c_{s}(\mathbbm{1}^{\wedge})$
 counts the number of alternating permutations on $s-1$ elements. This statistic is given by the Euler number $E_{s-1}$. 
The first few values of $c_{s}(\mathbbm{1}^{\wedge})$ are $0,-1,0,2,0,-16,0,272, 0, -7936,\ldots$
 
 The connections between Euler numbers, Bernoulli numbers, generating functions, etc. are well known. In particular, in Stanley's comprehensive text we find the very elegant exponential generating function \cite[Prop.~1.6.1]{Stanley}:
\[
\sum_{n\geq 0}E_{n}\frac{x^{n}}{n!}  = \tan(x) + \sec(x)
,\]
and taking the odd terms of this function, we have our result.
\end{proof}
\begin{remark}
Note the following connection between symmetrized cycle-sums and principal minors for skew-symmetric matrices: Since $d_{s}(\mathbbm{1}^{\wedge}) = 1$, we obtain another formula for $c_{s}$, and thus for $E_{n}$ by the change of coordinates in Prop.~\ref{thm:symm-d_s=c_s}.
\end{remark}

\begin{theorem} Suppose $n\geq 4$
and let $\J^{\wedge}_{n}$ denote the ideal of the variety of (even sized) symmetrized cycle-sums for a generic skew-symmetric matrix $A\in \bw{2}\CC^{n}$.
$\J^{\wedge}_{4}$ decomposes as the intersection of two prime components
\[\J^{\wedge}_{4} \;=\;
\langle
-2 c_{2}^{2}+c_{4}
\rangle\quad \cap \quad \langle
-6 c_{2}^{2}-c_{4}
 \rangle.\]
When $n\geq 5$ 
$\J^{\wedge}_n$
is the prime ideal
\[
\langle E_{2(i+j)-1}c_{2i}c_{2j} -E_{2i-1}E_{2j-1} c_{2(i+j)} \mid 1\leq i \leq j \leq \lfloor n/2 \rfloor
\rangle.\]
\end{theorem}
\begin{proof} 
We first pull back the symmetrization of cycle-sums condition to the space of skew-symmetric matrices. Set-theoretically this produces skew-symmetric matrices of special format, as described in Theorem~\ref{thm:SCS-skew-symmetric}. We then consider the parametrizations of cycle-sums of these different types of matrices producing (possibly) different components in the ideal of relations among symmetrized cycle sums.

Let $A$ be a skew-symmetric $n\times n$ SCS matrix. If $c_{2}=0$, then $A$ is the zero matrix. This contributes to the embedded component $ \langle c_{2},\dots,c_{2\lfloor n/2 \rfloor}  \rangle$. 
This radical ideal clearly contains the proposed $\J^{\wedge}_{n}$. 
So assume $A \neq 0$.

So for all $i,j$ we have $a_{i,j} = \pm b$ for some fixed $b\in \CC$.  In particular, when $n=4$, $C_{\{1,2,3,4\}}$ must be one of $\{2^{4}, - 6b^{4}\}$, depending on the sign of the three oriented 4-cycles. So either $c_{4} = 2c_{2}^{2}$ or $c_{4}= 6c_{2}^{2}$. An easy check in Macualay2 verifies that when $n=4$, $\J^{\wedge}_4$ decomposes as the intersection of two prime components one corresponding to each of these possibilities:
\[\J^{\wedge}_{4} \;=\;
\langle
 c_{1},c_{3},-2 c_{2}^{2}+c_{4}
\rangle\quad \cap \quad \langle
c_{1},c_{3},-6 c_{2}^{2}-c_{4}
 \rangle.\]
The case $n\geq 5$ follows from Theorem~\ref{thm:SCS-skew-symmetric}, which says that $A$ is conjugate to $b\cdot \mathbbm{1}^{\wedge}$ for some nonzero $b\in \CC$ and Lemma~\ref{skewParam}, which says that $c_{k}(A) = b^{k}E_{k-1}$.
So it is clear that the claimed generators are in the ideal $\J^{\wedge}_{n}$. 
Now suppose there is an $f$  in $\J^{\wedge}_{n}$, and let $k$ be the largest integer such that $c_{k}$ occurs in a monomial of $f$. Using the equations we already have, we can replace all $c_{k}$ by monomials in variables $c_{i}$ with $i<k$. After this substitution we may assume that $f\in \CC[c_{1},\ldots,c_{5}] \cap \langle c_{1},c_{3},c_{5}\rangle$, i.e. $f$ is then a relation between $c_{2}(B)$ and $c_{4}(B)$ for $B$ a generic skew-symmetric  $5\times 5$ SCS matrix. So $f$ must be in the principal ideal $\langle 2c_{2}^{2}-c_{4} \rangle$. Thus $f$ is in the ideal generated by the claimed set of generators. 
\end{proof}

\subsection{Arbitrary square matrices}\label{sec:nonsym relations}
Recall that in the general case that $c_{2}c_{3}\neq 0$ every $n\times n$ matrix is
a diagonal matrix plus a scalar multiple of a conjugate of the Toeplitz matrix $T_{n}(x)$. 
The principal minors of the Toeplitz matrix $T_{n}(x)$ are easy to calculate:
\begin{prop} Suppose $x\neq 0,i$.
For $2\leq s \leq n$ the principal minors of $T_{n}(x)$ are 
\[
d_{s} (T_{n}(x))
= \frac{1}{x^{s-2}} \sum_{i=0}^{s} (-1)^{i } x^{2(s-i)}
=
\frac{(x^{2})^{s-1}+(-1)^{s}}{x^{s-2}(x^{2}+1)}
,\]
and 
\[
(x^{2}+1)d_{s}(x\cdot T_{n}(x)) =
  x^{2}((x^{2})^{s-1}+(-1)^{s})
.\]
\end{prop}

\begin{prop}\label{prop:c2c3algebraic}
Let $A$ be an $n\times n$ SCS matrix for $n\geq 5$, and suppose $c_{2}c_{3}\neq 0$.  Then $c_{k}(A)$ is an algebraic function of $c_{2}$ and $c_{3}$.
\end{prop}
\begin{proof}
Note that it suffices to prove the proposition for $T_{n}(x)$ since we may re-scale $A$ by $\sqrt{-c_{2}}$ and conjugate the result to obtain $T_{n}(x)$ producing cycle-sums that satisfy $c_{k}(A) = (-c_{2})^{k/2}c_{k}(T_{n}(x))$. Then the result follows by using the recursion satisfied by the Eulerian polynomials. We may also solve the initial case relating $c_{2}$ and $c_{3}$ and then use the formula for $c_{k}(T_{n}(x))$.

Since $c_{3}(T_{x}) = x-\frac{1}{x}$, we can solve for $x$ to obtain two roots,
\[   \left\{ \frac{c_3-\sqrt {{c_3}^{2}+4}}{2}, \frac{c_3+\sqrt {{c_3}^{2}+4} }{2}\right\}.
\]
Similarly, for  $c_{3}(A) = (-c_{2})^{\frac{3}{2}}(x-\frac{1}{x})$, we find the two roots:
\[ \left\{ \frac {c_3+\sqrt {-4\,{c_2}^{3}+{c_3}^{2}}}{ 2\left( -c_2 \right) ^{3/2}},\frac {c_3-\sqrt {-4\,{c_2}^{3}+{c_3}^{2}}}{ 2\left( -c_2 \right) ^{3/2}} \right\} 
.\]
In particular, after choosing one of these roots, we can express $A$ depending algebraically on the two parameters $c_{2}$ and $c_{3}$, and since $c_{k}(A)$ is a polynomial function in $A$, we have our result.
\end{proof}
\begin{cor}Suppose $c_{2}c_{3}\neq 0$. Suppose $A = T_{n}(x)$, with $x$ equal to one of the values $  \frac{c_3\pm \sqrt {{c_3}^{2}+4}}{2}$.  Then $A$ has $c_{2}(A)=\cdots = c_{n}(A)=0$.
Thus the cycle-sum map restricted to matrices of this form parametrizes a scheme corresponding supported on 
$
\V(\langle c_{2},\dots,c_{n}  \rangle)
.$
\end{cor}
\begin{remark}
Theorems \ref{thm:c_2=0}, \ref{thm:c_3=0}, \ref{thm:c_2!=0,c_3!=0},  and Proposition \ref{prop:c2c3algebraic} solve the symmetrized cycle-sum assignment problem and hence also the symmetrized principal minor assignment problem.  To determine whether there is an $n\times n$ matrix with cycle-sums $(c_{1},\ldots,c_{n})$, first consider the three cases depending on the vanishing / non-vanishing of $c_{2}$ and $c_{3}$. In each case, we give a recipe for constructing a matrix with the prescribed cycle-sums $c_{2}$ and $c_{3}$, it suffices to see whether the values of $c_{k}$ for $3<k\leq n$ agree with those for the matrix we construct.
\end{remark}

The following result describes the ideal of relations among symmetrized cycle-sums of arbitrary square matrices up to saturation.

\begin{theorem}\label{thm:nonsym} Let $n\geq 3$, let  $\J_{n}$ denote the ideal of the variety of symmetrized cycle-sums of $n\times n$ matrices. 
$\J_{3}$ is empty.
$\J_{4}$ decomposes as the intersection of two prime components:
\[\langle 2c_2^3+c_3^2-c_2c_4\rangle\qquad \text{and} \qquad \langle c_3,6c_2^2+c_4 \rangle.\]
When $n\geq 5$, $\J_{n}$ consists of
a determinantal component with codimension $n-3$ and degree $\binom{n-1}{2}$, generated by the maximal minors of the following matrix:
\begin{equation}\label{eq:Dmatrix}
\begin{pmatrix}
d_{0}       &  d_{1}       & d_{2} &\dots & d_{n-2}\\	
d_{1}       &  d_{2} &d_{3} &\dots &d_{n-1} \\
d_{2} & d_{3}& d_{4} &\dots & d_{n}\\
\end{pmatrix}
,\end{equation}
and the line
\[
\langle c_{2},\dots,c_{n}  \rangle
.\]
\end{theorem}

\begin{proof}[Proof of Theorem \ref{thm:nonsym}] 
The cases $n=3,4$ may be verified directly using elimination and symbolic algebra software such as \texttt{Macaulay2}. The degree and codimension of this determinantal ideal are well-known facts about determinantal varieties.

Suppose $n\geq 5$ and $A$ has symmetrized cycle-sums, and consider the different types of matrices that may be conjugate to $A$.
 First let us consider the case when $c_{1}(A)=0$. 
 Note that the case when $A$ is strictly upper-triangular is handled separately, and the ideal of relations among cycle-sums in that case is (up to radical) the line $\langle c_{2},\dots,c_{n}\rangle$. 

On the other hand, when $A$ is conjugate to an $n$-cycle matrix, we will see that its principal minors satisfy the equations for the general case when $c_{2}c_{3}\neq 0$, so the image of $n$-cycles is contained in the closure of the image of scalar multiples of the Toeplitz matrix $T_{n}(x)$. 

If $c_{2}\neq 0$ and $c_{3}=0$, then $A$ is conjugate to $T_{n}(1) = \mathbbm{1}^{\wedge}_{n}$, so this case is also included in the $c_{2}c_{3} \neq 0$ case.

Now suppose $c_{2}c_{3} \neq 0$ and $x\neq 0,\pm i$.  
 Then $A$ is conjugate to a scalar multiple of $T_{n}(x)$, where we consider $x$ to be an arbitrary parameter.  $T_{n}(x)$ has principal minors equal to \[
d_{s} (T_{n}(x))
=
\frac{(x^{2})^{s-1}+(-1)^{s}}{x^{s-2}(x^{2}+1)}
\]
Recall the result of diagonal modification:
\[ d_{k}(A-\lambda I) = \sum_{i=0}^{k} \binom{k}{i}(- \lambda)^{k-i}d_{i}(A)\]

So
\[ d_{k}(xT_{n}(x) + y I) = \sum_{i=0}^{k} \binom{k}{i}(y)^{k-i}d_{i}(xT_{n}(x))
=
\sum_{i=0}^{k} \binom{k}{i}(y)^{k-i} 
\frac{(x^{2})^{i}+(-1)^{i}x^{2}}{(x^{2}+1)}
\]
Clearing denominators and using the binomial theorem, we have,
\[
(x^{2}+1)d_{k}(xT_{n}(x) + y I) = 
\sum_{i=0}^{k} \binom{k}{i}(y)^{k-i} 
((x^{2})^{i}+(-1)^{i}x^{2})
\]
\[
=
\sum_{i=0}^{k} \binom{k}{i}(y)^{k-i} 
(x^{2})^{i} + \sum_{i=0}^{k} \binom{k}{i}(y)^{k-i} (-1)^{i}x^{2} 
\]
\[
= (x^{2} + y)^{k} + (y-1)^{k}x^{2}  .
\]

Finally, 
\[
d_{k}(xT_{n}(x) + y I)  = \frac{(x^{2} + y)^{k} + (y-1)^{k}x^{2}  }{1+x^{2}}
.\]

Therefore, the image of the principal minors map is
\[
v = \left[
\frac{1+x^{2}}{1+x^{2}}, 
\frac{(x^{2} + y) + (y-1)x^{2}  }{1+x^{2}}
,
\frac{(x^{2} + y)^{2} + (y-1)^{2}x^{2}  }{1+x^{2}},
\ldots,
\frac{(x^{2} + y)^{n} + (y-1)^{n}x^{2}  }{1+x^{2}}
\right]
.\]
Notice that $v = u + w$, where $u =\left [\frac{(x^{2}+y)^{k}}{1+x^{2}}\right]_{k=0}^{n}$ and $w =\left[ \frac{ (y-1)^{k}x^{2} }{1+x^{2}}\right]_{k=0}^{n}$.  For all $x\neq \pm i$ we have $u \in \nu_{n}\PP^{1}$, and for all $x\neq 0,\pm i$ we have $w \in \nu_{n}\PP^{1}$. Since $(x^{2}+y)$ and $(y-1)$ are independent for $x\neq \pm i$, we can parametrize an open subset of $\sigma_{2}\nu_{n}\PP^{1}$ as the vectors of principal minors of $xT_{n}(x) + y I$ as $x,y$ vary. 
The prime ideal of $\sigma_{2} \nu_{n}\PP^{1}$ is well-known to be generated by the $3\times 3$ minors of the (catalecticant) matrix
\begin{equation}\label{eq:Dmatrix}
\begin{pmatrix}
d_{0}       &  d_{1}       & d_{2} &\dots & d_{n-2}\\	
d_{1}       &  d_{2} &d_{3} &\dots &d_{n-1} \\
d_{2} & d_{3}& d_{4} &\dots & d_{n}\\
\end{pmatrix}.
\end{equation}

Now while it is tricky (because of the inhomogeneity) to take a limit of matrices of the form $xT_{n}(x) + yI_{n}$ to produce a matrix corresponding to a weighted $n$-cycle, the principal minors of a weighted $n$-cycle form the vector $[1,0,\ldots,0,z]$, for $z \in \CC$. This point is clearly an element of $\sigma_{2} \nu_{n}\PP^{1}$. So the vectors of principal minors of the weighted $n$-cycle matrix are zeros of the equations given by the $3\times 3$ minors of matrix \eqref{eq:Dmatrix}.
\end{proof}

A consequence of the proof of the previous result is the following geometric characterization of the variety of principal minors of SCS matrices:
\begin{cor} For $n\neq 4$, the scheme of cycle-sums of $n\times n$ SCS matrices corresponds to the chordal variety of the rational normal curve together with a non-trivial embedded scheme supported on the line $\mathcal{V}( c_{2},\dots,c_{n})$. As sets
\[ Z_{n}\cap  S^{n}\CC^{2} = 
(\sigma_{2}\nu_{n}\PP^{1} \cap  U_{c_{0}=1}) \cup \mathcal{V}( c_{2},\dots,c_{n}).\]
\end{cor}
\begin{remark}\label{rmk:structure}
The scheme structure of $\mathcal{V}( c_{2},\dots,c_{n})$ is quite complicated, and highly non-linear in principal minor coordinates. For example, when $n=5$
we found the relations among cycle-sums  and the relations among principal minors by performing a standard elimination procedure using \texttt{Macaulay2} \cite{M2}. In the cycle-sum case, we found that 
the ideal is generated by 1 cubic, 8 quartics and 1 quintic equation:
\[\begin{array}{l}
3c_3^3-4c_2c_3c_4+c_2^2c_5,\\
 6c_2c_3^2c_5-2c_2^2c_4c_5+c_4^2c_5-c_3c_5^2,\\
 6c_2^2c_3c_5+c_3c_4c_5-c_2c_5^2, 2c_2^3c_5+c_3^2c_5-c_2c_4c_5,\\
 6c_2c_3^2c_4-2c_2^2c_4^2+c_4^3-c_3c_4c_5,\\
 6c_2^2c_3c_4+c_3c_4^2-c_2c_4c_5, 2c_2^3c_4+c_3^2c_4-c_2c_4^2,\\
 6c_2^2c_3^2+c_3^2c_4-c_2c_3c_5, 6c_2^3c_3+c_2c_3c_4-c_2^2c_5,\\
 12c_2^5+2c_3^2c_4-3c_2c_4^2+c_2c_3c_5.
 \end{array}\]
This computation took 4644.28 seconds on a server with 24 1.6GHz processors and 141GB of RAM (not all processors or all the memory are used at all times in M2).
The primary decomposition consists of one component of codimension 2 and degree 6 given by
\[\langle
3c_3^3-4c_2c_3c_4+c_2^2c_5,6c_2c_3^2-2c_2^2c_4+c_4^2-c_3c_5,6c_2^2c_3+c_3c_4-c_2c_5,2c_2^3+c_3^2-c_2c
 _4\rangle\]
 and another component of codimension 4 and degree 11 given by
 \[ \langle c_5,c_2c_4,c_3^2,c_4^3,c_3c_4^2,c_2^3c_3,c_2^5 \rangle,
\quad \text{ with radical } \quad 
 \langle 
 c_{2},c_{3},c_{4},c_{5}
 \rangle
 .\]
Whereas the ideal of relations among symmetrized principal minors is much more complicated; it is generated by 20 quintics and 13 sextics,
and decomposes as
one component with codimension 2 and degree 6, generated by 4 cubics,
and another component with codimension 4 and degree 75, with 47 generators of degrees up to 8. \looseness=-1

\end{remark}

\section*{Acknowledgements}
Oeding thanks Bernd Sturmfels for introducing us to this question, and for his continued excellence in mentorship.  Oeding is also grateful for the partial support provided by the South Korean National Institute for Mathematical Sciences (NIMS) where some of this work was carried out.  The authors are also grateful to the developers of \texttt{Macaulay2}, where the initial examples in this paper were all computed.  The authors are also thankful for the careful review of two referees whose remarks improved the exposition of this work, and helped simplify the conditions on Thm.~1.2(1).
\newcommand{\arxiv}[1]{\href{http://arxiv.org/abs/#1}{{\tt arXiv:#1}}}

\def\Dbar{\leavevmode\lower.6ex\hbox to 0pt{\hskip-.23ex \accent"16\hss}D}
  \def\cprime{$'$}
\providecommand{\bysame}{\leavevmode\hbox to3em{\hrulefill}\thinspace}
\providecommand{\MR}{\relax\ifhmode\unskip\space\fi MR }
\providecommand{\MRhref}[2]{%
  \href{http://www.ams.org/mathscinet-getitem?mr=#1}{#2}
}
\providecommand{\href}[2]{#2}

\end{document}